\documentclass[11pt]{article}
\usepackage[centertags,intlimits]{amsmath}

\usepackage{amsfonts}                     
\usepackage{amsthm}
\usepackage{graphics}

\allowdisplaybreaks[2]
\numberwithin{equation}{section}
\newtheorem{theorem}{Theorem}[section]
\newtheorem{lemma}{Lemma}[section]

\theoremstyle{remark}
\newtheorem{remark}{Remark}[section]

\providecommand{\abs}[1]{\lvert #1\rvert}
\providecommand{\norm}[1]{\lVert #1\rVert}

\DeclareMathOperator*{\esssup}{ess\,sup}

\newcommand{\nc}{\newcommand}
\nc{\vb}{\mathbf{v}}
\nc{\bx}{\mathbf{x}}
\nc{\by}{\mathbf{y}}
\nc{\bz}{\mathbf{z}}
\nc{\bu}{\mathbf{u}}
\nc{\bv}{\mathbf{v}}
\nc{\ba}{\mathbf{a}}
\nc{\bs}{\mathbf{s}}
\nc{\bq}{\mathbf{q}}
\nc{\bd}{\mathbf{d}}
\nc{\bb}{\mathbf{b}}
\nc{\bc}{\mathbf{c}}
\nc{\bi}{\mathbf{i}}
\nc{\bfr}{\mathbf{r}}
\nc{\bA}{\mathbf{A}}
\nc{\R}{\mathbb R}
\nc{\N}{\mathbb N}
\nc{\C}{\mathbb C}
\nc{\D}{\mathbb D}
\nc{\Z}{\mathbb Z}
\nc{\F}{\mathbf F}
\nc{\bbS}{\mathbb S}
\nc{\bE}{\mathbf E}
\nc{\B}{\cal B}
\nc{\br}{\bigr}
\nc{\bl}{\bigl}
\nc{\Bl}{\Bigl}
\nc{\Br}{\Bigr}
\nc{\ind}[1]{\,\mathbf{1}_{\{#1\}}\,}
\nc{\bP}{\mathbf{P}}

\title{A large-population limit for a Markovian model of group-structured populations} 
\author{A. Puhalskii,  B. Simon}

\begin{document}
\maketitle




\begin{abstract}
A Markovian model of group-structured (two-level) population dynamics features births, deaths, and migrations of individuals,
and fission and extinction of groups. These models are useful for studying group selection and other evolutionary processes that
occur when individuals live in distinct groups.
We show that the sample paths of a properly scaled sequence of these
models converge in an appropriate Skorohod space 
to a deterministic trajectory that is a unique solution to
 a quasilinear evolution equation. 
The  PDE model
can therefore be justified as an approximation to the Markovian one.
\end{abstract}
\section {Introduction}
A group-structured population is one comprised of individuals that live in distinct groups.
The individuals may be of different types, e.g., ``cooperators'' and ``defectors'', based on different heritable traits. 
Group-structured populations are very common in the natural world. Social animals often live in distinct groups, 
e.g., packs of wolves, ant colonies, tribes of hunter-gatherers; and group structure exist in other populations as well,
e.g., the parasites that reside in a given host organism constitute a group. The individuals in a group-structured population are not necessarily 
multicellular organisms. For example, the individuals could be microbes, like Dictyostelium (a species of amoeba), 
in which case the corresponding groups would be the  multicellular organisms made of them (slime molds).

Evolutionary theorists since Darwin's time have wondered how group structure may (or may not) affect evolution, 
and in particular the evolution of cooperative (altruistic) behaviors. 
Writing about prehistoric tribes of hunter-gatherers, 
Darwin \cite[p. 166]{Dar1871} famously said 
\begin{quotation}
\noindent {\it\small There can be no doubt that a tribe including many members who, from possessing in high degree the spirit of patriotism, fidelity, obedience, 
courage, and sympathy, were always ready to aid one another, and to sacrifice themselves for the common good, would be victorious over most other tribes;
and this would be natural selection.}
\end{quotation}
In other words, although altruistic behavior may be reproductively disadvantageous for individuals (the hunter-gatherers) within the groups (the tribes), 
groups with more altruistic individuals fare better, and over time this could lead to the evolution of altruistic individuals in the population as a whole.
Altruistic behavior cannot easily evolve in a well-mixed population\footnote{If the individuals are sufficiently closely related then
Hamilton's rule \cite{Ham64} shows that altruistic traits can evolve.}
since altruists can be exploited by nonaltruists, thereby lowering their ``fitness'', 
but it may evolve more easily in group-structured populations, where more-cooperative groups have some survival advantage.

By the 1960's Darwin's thoughts on this topic had been rephrased in terms of the efficacy of group selection.
The best arguments and mathematical models at that time, e.g.,
Maynard Smith \cite{May64}, Williams \cite{Wil66},   implied that group selection,
while theoretically possible, was much too weak to have an effect in the natural world. 
Explanatory theories like ``the selfish gene'', Williams \cite{Wil66}, Dawkins \cite{Daw76} and ``kin selection'', Hamilton \cite{Ham64}, Maynard Smith \cite{May64}, 
reinforced those arguments.  
Later, Wilson \cite{Wil75}, Sober and Wilson \cite{Unto98}, Wilson and Wilson \cite{W&W07}, and others gradually reopened the debate, 
being more careful about the way group selection was defined. 
The new group selection proponents realized that group selection (properly defined) was a potent evolutionary force, 
however they lacked mathematical models to help make their cases.
The problem with all the mathematical models of group selection up to that point was that they did not
properly account for group-level birth and death events like fission and extinction. Furthermore, the models were typically designed
for an equilibrium analysis only, so the time-dependent mechanisms that led to the evolution of cooperation/altruism remained mysterious.
The later proponents of group selection were right about the efficacy of group selection,
but they did not have a good mathematical explanation of how the phenomenon works.   

In Simon \cite{Sim10} a Markovian model of group-structured population dynamics was proposed that featured individual-level events like births, deaths,
and migrations, and also group-level events like fission and extinction. 
The paper also included a heuristic derivation of a PDE based on the Markovian model. 
The population dynamics from the Markovian model and PDE model are very similar to each other (Figure 1), 
leading one to suspect that there is a fundamental mathematical connection between the two models.
The purpose of this paper is to make the connection precise. 
The PDE trajectory is proven to be the unique limit of a certain scaled sequence of sample paths from the Markovian model.

Numerical experiments with the Markovian and PDE models from Simon \cite{Sim10} showed that previous models that ignored or misrepresented group-level 
birth and death events vastly underestimated the strength of group selection, Simon, Fletcher, Doebeli \cite{SFD13}, Simon and Pilosov \cite{S&P16}. 
There are other contemporary models of group-structured populations analyzed in the literature that also shed light on the process of group selection. 
Traulsen and Nowak \cite{T&N06} proposed a model based on nested Moran processes.
They found conditions for a single mutant cooperator in a population of defectors to have a better chance of fixing in the population than a single mutant defector
in a population of cooperators. Their results are valid in a limiting regime as the rates of group-level births and deaths approach zero,
and as the fitness advantage for defectors approaches zero (weak selection).
Another difference between their work and ours is that by looking only at fixation probabilities, their analysis did not involve population dynamics.
Luo \cite{Luo13} and Luo and Mattingly \cite{LuoMat17} also considered nested Moran processes as models of group-structured populations. 
They obtained large--population limits for the group process.
One limit is the solution of 
 a deterministic PDE and the other one is a Fleming--Viot process. 
A wide variety of large--population limits for Markovian models 
of a similar kind, but restricted to a single biological level (no group structure), 
can be found in Champagnat, et.al. \cite{Champ06} and Puhalskii and Simon \cite{P&S12}. 

The present results are the same sort as the PDE limit in Luo and Mattingly \cite{LuoMat17}, but the model of group-structured populations
studied here is more (biologically) realistic. In particular, the birth and death events at each level are decoupled here, 
i.e., it is not necessary to assume that the number of groups and the sizes
of the groups are constant, which is a defining feature of the Moran
process.               
Furthermore, the group-level birth event in the  present model is fission, 
which is more realistic than the ``group cloning'' birth
event implicit in a model based on Moran processes.
Migrations are introduced into the model 
and the number of types of  individuals is not restricted to two,
unlike in Luo and Mattingly \cite{LuoMat17}.
On the technical side, abandoning the conventions of the Moran model
forces one to grapple with convergence of processes assuming values in
noncompact spaces of measures and accounting for migrations 
results in the limit equation being quasilinear whereas the PDE
in Luo and Mattingly \cite{LuoMat17} is semilinear.
In addition, rather than working with the generator of the Markov
process of the number of groups which seems to be problematic in our
setup, we work directly with the balance equations for 
the process trajectories. We prove convergence
to the limit PDE for two different topologies on the state--space of
the population process. One is the topology of
weak convergence of nonnegative measures akin to that used in 
 Luo and Mattingly \cite{LuoMat17} 
and the other is a stronger  topology of  weak
 convergence in an $\mathbb L^p$--space. The second mode of
 convergence  stipulates
a different scaling.

\section {The Markovian model of group-structured populations}

We consider a population  of a finite number of groups, where each group consists of a finite number of individuals. 
There are $\ell$
 types of individuals, e.g., cooperators and defectors.
The state of a group is specified by the number of each type in the group. 
An $i=(i_1,\ldots,i_\ell)$--group is a group with $i_1$ type 1
individuals,
 $i_2$ type 2 individuals and so on. 
(We treat the $i$ as  $\R^\ell$--vectors, e.g.,
$i-i'=(i_1-i_1',\ldots,i_\ell-i_\ell')$\,, and define $i\ge i'$ to
mean that the entries of $i-i'$ are nonnegative.
We denote $\abs{i}=i_1+\ldots+i_\ell$ and
  let
$e_k$ represent the $k$th element of the standard basis in
$\R^\ell$\,,
e.g., $e_1=(1,0,\ldots,0)$\,.)
Within each group, individuals independently 
give birth (asexually and without mutation) and die at stochastic rates 
which may depend on the individual type and the state of the group.
 The per capita birth rate of type $k$ individuals in an $i$-group is $\beta^k(i),~k=1,\ldots,\ell$.
Likewise, the per capita death rate of type $k$ individuals in an $i$-group is $\delta^k(i)$\,.
(Naturally, we assume that $\beta^k(i)=\delta^k(i)=0$ when $i_k=0$\,.)

Let $X_t(i)$ be the number of $i$--groups in the population at time $t$. 
Then $\{ X_t(i),~i\in\mathbb Z_+^\ell\setminus\{0\}\}$ specifies the state of the population at time $t$. 
In the model,  the groups independently
 die of extinction, the extinction rate  for
 an $i$--group being   $\epsilon(i)X_t^\ast$, 
where $X_t^\ast = \sum_{i} X_t(i)$ is the total number of groups at time $t$. 
An extinction event is the instantaneous death of all the individuals
in a group, i.e., the death of the group. The other group--level event
in the model is fission.
Let us say that (unordered) 
set $\pi_i$ of  $\ell$--vectors with nonnegative integer
entries is a partition of $i$ if the vectors are nonzero and
$\sum_{j\in\pi_i}j=i$\,.
An $i$--group 
 fissions at rate $\phi(i)$ (independently of the
 states of the other groups) which means that
 $i$ is split according to
 partition $\pi_i$\,, which is chosen at random,
 the elements of the partition determine the makeup of
''offspring'' groups and the ``parent'' group ceases to exist. 
Thus, there is a conservation of individuals under fission.
(It is allowed for a partition to consist of the single vector $i$
which constitutes ''a nonproper  fission''.)
The probability that an $i$--group is fissioned according to
partition $\pi_i$  is denoted by $\zeta_i(\pi_i)$\,.
Evidently, $\sum_{\pi_i}\zeta_i(\pi_i)=1$\,. 
Let $\pi_i(i')$ denote the number of $i'$--groups in partition
$\pi_i$\,.
We let $\eta(i,i')=\sum_{\pi_i}
\pi_i(i')\zeta_i(\pi_i)$ denote the expected number of $i'$--groups
produced by the fission.
    Finally, individuals can independently 
migrate from one group to another in the
model. The per capita migration rate of type $k$ individuals in an
$i$--group  is $\mu^k(i)$\,. It is assumed that a
migrating individual chooses a group from the population to join
(possibly, the one they are coming from), 
each with equal probability, considering themselves as a member of their
group when deciding on the move.

\medskip

Let
\begin{trivlist}
\item 
$B_t^{k}(i)$   represent the number of type $k$  births in all
$i$--groups in $[0,t]$\,, where
$B_t^{k}(i)=0$ when  $i_k=0$\,,
\item 
$D_t^{k}(i)$   represent the number of type $k$ deaths
 in all
$i$--groups in $[0,t]$\,, where
$D_t^{k}(i)=0$ when  $i_k=0$\,,
\item 
$M_t^{k}(i)$   represent the 
 number of type $k$  immigrations 
 to all
$i$--groups in $[0,t]$\,,
\item
$\overline M_t^{k}(i)$  represent the 
 number of type $k$  emigrations 
 from all
$i$--groups in $[0,t]$\,, where
$\overline M_t^{k}(i)=0$  when $i_k=0$\,,
\item
$\overline F_t(i)$ represent the number of fissions
 of $i$--groups in $[0,t]$\,, 
\item
$F_t(i',i)$ represent the number of
$i$--groups that are produced as a result of fissioning
 of $i'$--groups in $[0,t]$\,,
\item
$E_t(i)$ represent the number 
of $i$--groups that get extinct in $[0,t]$\,.
\end{trivlist}
It is assumed that all these processes  take
values in $\mathbb Z_+$\,, are equal to zero when $t=0$ and   have nondecreasing
piecewise constant rightcontinuous trajectories with lefthand
limits. All the processes with the exception of $F_t(i',i)$ have unit
jumps. The jump sizes  of the
 process $F_t(i',i)$ are determined by the numbers of groups that 
the $i'$--groups  may fission into.
The birth, death,  migration, and extinction 
processes are assumed to be independent
Poisson processes. 
In order to be more specific, we
introduce  independent ''primitives''  as follows.
Let, 
for $i\in\mathbb Z_+^\ell\setminus\{0\}$\,, $k\in\{1,\ldots,\ell\}$\,, $p\in\N$\,,
$l\in\N$\,, 
  $r\in\N$\,, and  $X=(X(i))\in \Z_+^{\mathbb Z_+^\ell
\setminus\{0\}}$ with $X^\ast=\sum_{i}X(i)\in(0,\infty)$\,,
\begin{trivlist}
\item $L^{B,k}_t(i,p,r)$\, 
 represent Poisson processes of rates 
$\beta^{k}(i)$\,,
\item $L^{D,k}_t(i,p,r)$\, 
 represent Poisson processes of rates 
$\delta^{k}(i)$\,,
\item $L^{\overline M,k}_t(i,p,r)$\, 
 represent Poisson processes of rates 
$\mu^{k}(i)$\,,
\item $L^{\overline F}_t(i,p)$
 represent Poisson processes of rates 
$\phi(i)$\,,
\item $L^{ E}_t(i,p)$
 represent Poisson processes of rates 
$\epsilon(i)$\,,
\item $\vartheta^k_i(p,r,X)$\,, with $i_k\ge 1$\,,
  represent 
  random variables assuming values in $\mathbb
  Z_+^\ell\setminus\{0\}$ such that
  \begin{equation*}
\mathbf P\bl(\vartheta^k_i(p,r,X)=i'\br)=
\frac{X(i')}{X^\ast}\,,    
  \end{equation*}
\item $\theta_i(p)$ represent random partitions of $i$ distributed as 
$\zeta_i$\,. 
\end{trivlist}
(Informally, $p$ represents the index of a group and $r$ represents the
index of an individual in a group,
$\vartheta^k_i(p,r,X)$ represents the makeup of the group the $r$th
type $k$ migrating 
individual of the $p$th $i$--group joins upon migration when the
 state of the population is $X$\,, and $\theta_i(p)$ is the set of
 groups that the $p$th fissioning
 $i$--group fissions into.)
All these processes and random variables are assumed mutually
independent for different $i$\,, or $p$\,, or $r$\,.

We let 
$\theta_i(i',p)$ represent the number of
$i'$--groups in the partition $\theta_i(p)$  
 so that $\mathbf P(\theta_i(i',p)=j)=
\sum_{\pi_i:\,\pi_i(i')=j}\zeta_i(\pi_i)$\,.
It is noteworthy that
\begin{equation}
  \label{eq:97}
  \mathbf E\theta_i(i',p)=\eta(i,i')\,.
\end{equation}
We assume that the  $(X_t(i)\,,t\ge0)$ are
$\mathbb Z_+$--valued  processes which have
piecewise constant rightcontinuous trajectories with limits on the
left
 and that
the following recursions are satisfied, with  $\Delta$ representing
the jump of a process, with $t-$ denoting the lefthand limit so that, e.g.,
$\Delta X_t(i)=X_t(i)-X_{t-}(i)$ and
with $\mathbf 1_\Gamma$ denoting the indicator function of event or element
$\Gamma$\,:
  \begin{align}
         \label{eq:6}
       \begin{split}
    \Delta          B^{k}_t(i)=
\sum_{p=1}^{ X_{t-}(i)}\sum_{r=1}^{i_k}\Delta L^{B,k}_t(i,p,r)\,,\;
 \Delta          D^{k}_t(i)=
\sum_{p=1}^{ X_{t-}(i)}\sum_{r=1}^{i_k}\Delta L^{D,k}_t(i,p,r)\,,\\
\Delta          \overline M^{k}_t(i)=
\sum_{p=1}^{ X_{t-}(i)}\sum_{r=1}^{i_k}\Delta L^{\overline M,k}_t(i,p,r)
(1-
\mathbf 1_{\vartheta^k_i(p,L^{\overline M,k}_t(i,p,r),
X_{t-})}(i)),\\
\Delta           M^{k}_t(i)=
\sum_{i'\not=i}\sum_{p=1}^{
X_{t-}(i')}\sum_{r=1}^{i'_k}
\Delta L^{\overline M,k}_t(i',p,r)
\mathbf 1_{\vartheta^k_{i'}(p,L^{\overline M,k}_t(i',p,r),
X_{t-})}(i)\,,\\
\Delta          \overline F_t(i)=
\sum_{p=1}^{ X_{t-}(i)}\Delta L^{\overline F}_t(i,p)\,,\;
    \Delta F_t(i,i')=  \sum_{p=1}^{ X_{t-}(i)}
\theta_i(i',p)
\Delta L^{\overline F}_t(i,p)\,,\\
\Delta E_t(i)=
\sum_{p=1}^{X_{t-}(i)X_{t-}^\ast}
\Delta L^{ E}_t(i,p)
\end{split}\end{align}
and 
\begin{align}
  \label{eq:1}
  \begin{split}
      \Delta X_t(i)=-&\sum_{k=1}^\ell \Delta 
B_t^{k}(i)+\sum_{k=1}^\ell 
\Delta B^{k}_{t}(i-e_k)
-\sum_{k=1}^\ell \Delta D^k_t(i)+
\sum_{k=1}^\ell \Delta
D^{k}_t(i+e_k)\\
-&\sum_{k=1}^\ell \Delta M^{k}_{t}(i)
+\sum_{k=1}^\ell \Delta M^{k}_{t}(i-e_k)
-\sum_{k=1}^\ell \Delta \overline M^{k}_t(i)
+\sum_{k=1}^\ell \Delta \overline M^{k}_t(i+e_k)
\\-&
\Delta \overline{F}_t(i)+
\sum_{i'}\Delta F_t(i',i)-\Delta E_t(i)\,.
  \end{split}
\end{align}
(We also assume that $B^{k}_t(i)=M^{k}_t(i)=0$ when $i_k=-1$\,.)
Induction on the jump epochs of the primitive processes shows that
\eqref{eq:6} and  \eqref{eq:1}  admit a unique solution for given
$X_0$ up to the time that $X^\ast_t=0$\,, which may happen never. 
From that time on, we let
the righthand sides of \eqref{eq:6} vanish, so, $X_t(i)=0$\,.
If the $\beta^k(i)$ are bounded above uniformly in $i$ by some $\beta^k$\,,
then the population stays finite at all times, provided it is finite
initially,  because it does not exceed the
value of the Yule process with birth rate 
$\sum_{k=1}^\ell \beta^k$ and the population at time zero being
$X_0^\ast$\,.

Let $X_t=(X_t(i)\,, i\in \mathbb Z_+^\ell\setminus\{0\})$\,.
It is  a Markov process with values 
in $\mathbb Z_+^{\mathbb Z_+^\ell\setminus \{0\}}$\,. 
One can  view $X_t$ as a density with respect to the counting
measure on $\mathbb Z_+^\ell\setminus\{0\}$ so that $X_t$ can be
identified with the measure on $\mathbb Z_+^\ell\setminus\{0\}$ induced
by the density and defined by
$  \Lambda_t(\Gamma)=\sum_i X_t(i) \mathbf 1_\Gamma(i)
$\,, where $\Gamma\subset \mathbb Z_+^\ell\setminus\{0\}$\,.
 Then, $(X_t\,,t\ge0)$ can be referred to as a measure--valued
process.

For the limit theorem, we consider a sequence of models as above,
labelled with two  parameters, $m$ and $n$\,, which we let go to
infinity.
 Accordingly, the
variables we have introduced are supplemented with superscripts
$n$ and $m$\,, e.g., 
$B_t^{n,m,k}(i)$ stands for the number of type $k$
 births in all
$i$--groups in $[0,t]$ for the $(m,n)$--model.
Informally, $m$ characterises the group number and $n$ characterises
the group sizes.
It is assumed throughout that  the functions
$i_k(\beta^{n,m,k}(i)+\delta^{n,m,k}(i)+\mu^{n,m,k}(i))$\,, 
$\phi^{n,m}(i)$ and $m\epsilon^{n,m}(i)$
 are bounded in
$n,m,$ and $i$\,,  that the $m\epsilon^{n,m}(i)$ are bounded
away from zero, and that the number of groups that may be produced as
a result of fissioning is bounded, i.e.,
  the random variables $\theta^{n,m}_i(i',p)$ are
bounded. It is  convenient to extend the domain of $i$ to 
all of $\mathbb Z_+^\ell$ so that $\beta^{n,m,k}(0)$ and similar quantities
are defined, and define  $X^{n,m}_t(0)=0$\,.

Let 
\begin{equation}
  \label{eq:2}
    \Lambda^{n,m}_t(\Gamma)=\frac{1}{m}\,\sum_i X^{n,m}_t(i)\mathbf
    1_\Gamma
\bl(\frac{i}{n}\br)\,,
\end{equation}
where $\Gamma\subset\R_+^\ell$\,,
and let
$\Lambda^{n,m}=((\Lambda^{n,m}_t(\Gamma)\,,\Gamma\in\mathcal{B}(\R_+^\ell))\,,
t\ge0)$\,. Note that $\Lambda^{n,m}_t(\{0\})=0$\,.  The process
$\Lambda^{n,m}$ takes values in the space $\mathbb M_+(\R_+^\ell)$ of
(nonnegative finite) Borel measures on $\R_+^\ell$\,, which is equipped with the
weak topology and is, therefore, a complete separable metric space, see, 
Tops\oe \cite{Top}.
Accordingly,
$\Lambda^{n,m}$ is a random element of the Skorohod space $\mathbb
D(\R_+,\mathbb M_+(\R_+^\ell))$\,, see, e.g., Ethier and Kurtz
\cite{EthKur86} for the
definition and properties. 
We introduce a number of other spaces. 
Let $\mathbb C(\R_+,\mathbb
M_+(\R_+^\ell))$ denote the set of continuous $\mathbb
M_+(\R_+^\ell)$--valued functions.
Let $\mathbb C^1(\R_+^\ell)$ denote the
set of real--valued functions on $\R_+^\ell$ that can be extended to
functions with continuous derivatives 
defined on an open set containing $\R_+^\ell$\,.
Let $\mathbb
C^1_c(\R_+^\ell)$ denote the subset of $\mathbb C^1(\R_+^\ell)$ of
functions of compact support.

Let, for $u=(u_1,\ldots,u_\ell) \in\R_+^\ell$\,,
$\lfloor nu\rfloor=(\lfloor nu_1\rfloor,\ldots,\lfloor
nu_\ell\rfloor)$ and
\begin{align}
  \label{eq:5}
  \begin{split}
\hat\beta^{n,m,k}(u)=
\beta^{n,m,k}(\lfloor nu\rfloor),\;
\hat\delta^{n,m,k}(u)=
\delta^{n,m,k}(\lfloor nu\rfloor),\;
\hat\mu^{n,m,k}(u)=
\mu^{n,m,k}(\lfloor nu\rfloor),\\
\hat\phi^{n,m}(u)=
\phi^{n,m}(\lfloor nu\rfloor),\;
\hat\epsilon^{n,m}(u)=
m\epsilon^{n,m}(\lfloor nu\rfloor)\,.
  \end{split}
\end{align}
We also define, for $u\in\R_+^\ell$ and  $\Gamma\subset\mathbb R_+^\ell$\,,
\begin{equation}
  \label{eq:32}
\hat \eta^{n,m}(u,\Gamma)=
\sum_{i'}\eta^{n,m}(\lfloor nu\rfloor,
i')\mathbf 1_{\Gamma}\bl(\frac{i'}{n}\br)\,.
\end{equation} (Note that 
$\hat\eta^{n,m}(u,\{i'/n\})=\eta^{n,m}(\lfloor nu\rfloor,i')$\,.)

Let us assume that there exist 
 functions $\hat\beta^k(u)$\,, $\hat\delta^k(u)$ 
and $\hat\mu^k(u)$\,, which belong to 
$\mathbb C^1(\R_+^\ell)$\,, such that
 the functions
 $u_k\hat\delta^k(u)$\,, $u_k\hat\beta^k(u)$\,,
and $u_k\hat\mu^k(u)$ have bounded first order derivatives,
  continuous and bounded functions $\hat\phi(u)$ and
$\hat\epsilon(u)$ and
transition kernel $\hat\varphi(u,du')$\,, which
is a  finite   measure on $\R_+^\ell$   for each $u$\,, with the total mass being
 uniformly bounded over
 $u\in\R_+^\ell$\,,
such that 
$\int_{\R_+^\ell}f(u')\hat\varphi(u,du')$ is a continuous function of $u$
and, as 
$n,m\to\infty$\,,
\begin{multline}
  \label{eq:7}
  \sum_{k=1}^\ell\bl(u_k\abs{\hat\beta^{n,m,k}(u)-
\hat\beta^k(u)}+u_k\abs{
\hat\delta^{n,m,k}(u)-\hat\delta^k(u)}+u_k\abs{
\hat\mu^{n,m,k}(u)-\hat\mu^k(u)}\br)\\+
\abs{\hat\phi^{n,m}(u)-\hat\phi(u)}+
\abs{\hat\epsilon^{n,m}(u)-
\hat\epsilon(u)}
+\abs{
  \int_{\R_+^\ell}f(u')
\hat\phi^{n,m}(u)\hat\eta^{n,m}(u,du')\\-
\int_{\R_+^\ell}f(u')\hat\varphi(u,du')} \to0
\end{multline} uniformly over compact sets of $u$\,,
 for all  continuous bounded functions $f$ of compact support.
It is assumed further that the total population at time $0$\,,
which is 
$\sum_i\sum_{k=1}^\ell
i_kX^{n,m}_0(i)=nm\int_{\R_+^\ell}\abs{u}\Lambda^{n,m}_0(du)$\,, is
finite. (Accordingly, $nm\int_{\R_+^\ell}\abs{u}\Lambda^{n,m}_t(du)$
yields the total population at time $t$\,.)
Let $\abs{u}=u_1+\ldots+u_\ell$\,. 
\begin{theorem}
  \label{the:lln}
  \begin{enumerate}
  \item 
Suppose that,
 for some $\hat \lambda_0=(\hat \lambda_0(du))\in \mathbb 
M_+(\R_+^\ell)$ such that
 $\hat\lambda_0(\R_+^\ell)>0$ and
$\int_{\R_+^\ell}\abs{u}\hat\lambda_0(du)<\infty$\,, 
 we have that
$\Lambda^{n,m}_0\to \hat\lambda_0$   in probability in 
 $ \mathbb M_+(\R_+^\ell)$ and
$\int_{\R_+^\ell}\abs{u}\Lambda^{n,m}_0(du)\to
\int_{\R_+^\ell}\abs{u}\hat\lambda_0(du)$ in probability, 
  as $n,m\to\infty$\,. 
In addition, suppose that
\begin{equation*}
  \lim_{K\to\infty}
\limsup_{n,m\to\infty}
\mathbf P(\sum_i\abs{X^{n,m}_0(i)}^2
>Km^2)=0\,.
\end{equation*}
 Then, the  $\Lambda^{n,m}$ converge in
probability in $\mathbb D(\R_+,
 \mathbb M_+(\R_+^\ell))$   to 
 $(\hat\lambda_t\,, t\ge0)\in\mathbb C(\R_+,
\mathbb M_+(\R_+^\ell))$ such that  
$\hat\lambda_t(\R_+^\ell)>0$\,. It is uniquely specified by
the requirement that, given $f\in \mathbb C^1_c(\R_+^\ell)$\,, 
$\int_{\R_+^\ell}f(u)\hat\lambda_t(du)$ is differentiable
 and 
\begin{multline}
  \label{eq:13}
  \frac{d}{dt}\,\int_{\R_+^\ell}f(u)\hat\lambda_t(du)\\=
\int_{\R_+^\ell}\Bl(\sum_{k=1}^\ell 
\bl(u_k\,(\hat\beta^{k}(u)-\hat\delta^{k}(u)-\hat\mu^k(u))
+\frac{1}{
\hat\lambda_t(\R_+^\ell)}\,\displaystyle\int_{\R_+^\ell} 
u'_k\,\hat\mu^{k}(u')\,\hat\lambda_t(du')\br)
   \partial_{u_k} f(u)
\\+
\int_{\R_+^\ell}
f(u')\,\hat\varphi(u,du')
-(\hat\phi(u)+\hat\epsilon(u)
\hat\lambda_t(\R_+^\ell))f(u)\Br)\hat\lambda_t(du)\,.
\end{multline}
 If    $\hat\lambda_0$ is absolutely continuous  with
respect to Lebesgue measure,  then $\hat\lambda_t$ 
 is absolutely continuous too.
\item In probability, locally uniformly in $t$\,,
\begin{equation*}
        \int_{\R_+^\ell}\abs{u}\Lambda^{n,m}_t(du)
\to \int_{\R_+^\ell}\abs{u}\hat \lambda_t(du)\,.
\end{equation*}
  If, in addition, 
  \begin{equation*}
\int_{\R_+^\ell}u\,\Lambda^{n,m}_0(du)\to
\int_{\R_+^\ell}u\,\hat \lambda_0(du)   
  \end{equation*}
 in probability,
then
\begin{equation*}
       \int_{\R_+^\ell}u\,\Lambda^{n,m}_t(du)
\to \int_{\R_+^\ell}u\,\hat\lambda_t(du)
\end{equation*}
  in
probability locally uniformly in $t$\,.
\item
Suppose that, under the hypotheses of part 1, $\hat\lambda_0$
is absolutely continuous  with
respect to the Lebesgue measure, that its density
$\hat x_0=(\hat x_0(u)\,,u\in\R_+^\ell)$
is a bounded and Lipschitz--continuous function, that
$\hat\varphi (u,du')=\overline \varphi(u,u')\,du'$ with
$\overline \varphi(u,u')$ having Sobolev derivative $D_u\overline\varphi(u,u')$
with respect to $u$ for almost all $u'$ such that
$\esssup_{u\in\R_+^\ell}\int_{\R_+^\ell}
  (\overline\varphi(u,u')+\abs{D_u\overline\varphi(u,u')})\,du'<\infty$\,,
and that the functions 
$\hat\phi(u)$ and $\hat\epsilon(u)$ are  Lipschitz--continuous.
  Then the density
   $\hat x_t=(\hat x_t(u)\,,u\in\R_+^\ell)$ of $\hat\lambda_t$
  is a bounded and Lipschitz--continuous function of $u$, 
 is locally Lipschitz--continuous with
respect to $t$\,,  and, for  almost all $t$ and
  $u$ with respect to the Lebesgue measure,
\begin{multline}
  \label{eq:28}
-\partial_t  \hat x_t(u)
=\sum_{k=1}^\ell 
\Bl(\partial_{u_k}\bl(
\,\hat x_t(u)u_k\,(\hat\beta^{k}(u)-\hat\delta^{k}(u)-\hat\mu^k(u))\br)\\
+\frac{\displaystyle\int_{\R_+^\ell} \hat x_t(u')
u'_k\,\hat\mu^{k}(u')\,du'}{\displaystyle
\int_{\R_+^\ell}\hat x_t(u')\,du'}\,
   \partial_{u_k} \hat x_t(u)\Br)
-
\int_{\R_+^\ell}
\hat x_t(u')\overline\varphi(u',u)\,du'
\\+\hat x_t(u)\hat\phi(u)+\,\hat x_t(u)\hat\epsilon(u)
\int_{\R_+^\ell}\hat x_t(u')\,du'\,.
\end{multline}
  \end{enumerate}
\end{theorem}
\begin{remark}
  If $\hat\eta^{n,m}(u,du')\to \hat\eta(u,du')$ weakly, as in the next
  example, then one may be able to take
  $\hat\varphi(u,du')=\hat\phi(u)\hat\eta(u,du')$\,. 
\end{remark}
\begin{remark}
  As an example of the scaling, consider 
  fissioning into one or two pieces
that gives equal probability to every possible fission outcome:
\[
\zeta_i(\{i',i-i'\}) = \frac{1}{\prod_{k=1}^\ell (i_k+1)}\,.
\]
Hence, the $\theta_i(i',p)$ assume values in $\{1,2\}$ and
\begin{equation*}
  \eta(i,i') = \frac{2}{\prod_{k=1}^\ell (i_k+1)}
\,.
\end{equation*}
We define   fission kernel $\hat\eta^{n,m}(u,du')$ by
 \eqref{eq:32}.
Let $\phi^{n,m}(i)= \prod_{k=1}^\ell (i_k+1)  e^{-\abs{i}/n} /n^\ell$\,.
Consequently,
$\hat\phi^{n,m}(u)=
\prod_{k=1}^\ell  (\lfloor nu_k\rfloor+1)
e^{-\abs{\lfloor nu\rfloor}/n} /n^\ell$\,.
We have that
 $\hat\eta^{n,m}(u,du')\to2/\bl(\prod_{k=1}^\ell   u_k\br)\,
 \mathbf 1_{[0, u]}(u')\,du'$\,,
$\hat\phi^{n,m}(u)\to
\prod_{k=1}^\ell  u_k
e^{-\abs{u}}$\,, and
 $\hat\phi^{n,m}(u)\hat\eta^{n,m}(u,du')\to\\ 2
e^{-\abs{u}} \mathbf 1_{[0,u]}(u')\,du'
$ weakly uniformly over $u$ from bounded sets. 
 \end{remark}
We now give  a version of Theorem \ref{the:lln} for a stronger
topology.
Let
\begin{equation}
  \label{eq:46}
\check\epsilon^{n,m}(u)=
mn^\ell\epsilon^{n,m}(\lfloor nu\rfloor)\,,
\check\eta^{n,m}(u,u')=
n^\ell\eta^{n,m}(\lfloor nu\rfloor,\lfloor nu'\rfloor)\,.
\end{equation}
The functions $mn^\ell\epsilon^{n,m}(i)$ and 
$\phi^{n,m}(i)n^\ell\eta^{n,m}(i,i')$ are assumed to be bounded in
$n$\,, $m$\,, $i$\,, and $i'$\,.
Let us assume that there exist bounded
 Lipschitz--continuous  functions $\hat\beta^k(u)$\,,
 $\hat\delta^k(u)$\,, and
$\hat\mu^k(u)$\,, bounded continuous functions
 $\hat\phi(u)$ and $\check\epsilon(u)$\,, and function
$\check\varphi(u,u')$
such that
the functions
$u_k\hat\beta^k(u)$, $u_k\hat\delta^k(u)$\,, and 
$u_k\hat\mu^k(u)$ are bounded and Lipschitz--continuous, the function
$\int_{\R_+^\ell}f(u')\check\varphi(u,u')\,du'$ is continuous with
respect to $u$ for any continuous function $f(u')$ of compact support,
$\sup_{u\in\R_+^\ell}\int_{\R_+^\ell}\check\varphi(u,u')\,du'<\infty$\,, and
 for all bounded Borel measurable sets
$\Theta\subset \R_+^\ell$ and continuous functions of compact support $f(u')$\,,
\begin{align*}    
\begin{split}
\int_{\Theta}\Bl(\sum_{k=1}^\ell
\bl(u_k\abs{\hat\beta^{n,m,k}(u)-
\hat\beta^k(u)}+u_k\abs{
\hat\delta^{n,m,k}(u)-\hat\delta^k(u)}
+u_k\abs{
\hat\mu^{n,m,k}(u)-\hat\mu^k(u)}\br)\\+\abs{\hat\phi^{n,m}(u)-\hat\phi(u)}
+\abs{\check\epsilon^{n,m}(u)-
\check\epsilon(u)}
+
\abs{\int_{\R_+^\ell}
f(u')\hat\phi^{n,m}(u)\check\eta^{n,m}(u,u')\,du'\\-
\int_{\R^\ell_+}f(u') \check\varphi(u,u')\,du'}\,du\Br)\to0\,,
\end{split}\end{align*}
as $n,m\to\infty$\,.

Let processes
$      \hat X^{n,m}= ( \hat X^{n,m}_t    \,, t\ge0)$
  be defined by
  \begin{equation*}
\hat X^{n,m}_t=(\hat X^{n,m}_t(u),\, u\in \R_+^\ell) \text{ and }
    \hat X^{n,m}_t(u)=\frac{1}{m}\,X^{n,m}_t(\lfloor nu\rfloor)\,.
  \end{equation*}
We assume that $\hat X_0^{n,m}\in \mathbb L^2(\R_+^\ell)\cap \mathbb
L^1(\R_+^\ell)$\,. Both $\mathbb L^2(\R_+^\ell)$ and $\mathbb
L^1(\R_+^\ell)$ are endowed with the weak topologies. Specifically,
$\mathbb
L^2(\R_+^\ell)$ is endowed with the $\sigma\bl(
 \mathbb L^2,\mathbb 
L^2\br)$--topology and  
$\mathbb L^1(\R_+^\ell)$ is endowed with the $\sigma\bl(
 \mathbb L^1,\mathbb L^\infty\br)$--topology. Both spaces are
 completely regular topological spaces as topological groups.
We endow $\mathbb L^2(\R_+^\ell)\cap \mathbb L^1(\R_+^\ell)$ with the
weakest topology that is stronger than the restrictions of both weak topologies.
It is also a completely regular topological space, cf.,
Engelking \cite{Eng77}.  The Skorohod space
 $\mathbb D(\R_+, \mathbb L^2(\R_+^\ell)\cap \mathbb L^1(\R_+^\ell))$ 
is analysed in 
Jakubowski \cite{Jak86}.
\begin{theorem}
  \label{the:lln2}
  \begin{enumerate}
  \item 
If,
 for some $\check x_0=(\check x_0(u))\in 
 \mathbb L^2(\R^\ell_+)\cap \mathbb L^1(\R_+^\ell)$ with $\check x_0(u)\ge0$\,,
 $\int_{\R_+^\ell}\check x_0(u)\,du>0$\,, and
$\int_{\R_+^\ell}\abs{u}\check x_0(u)\,du<\infty$\,, 
 we have that
$\hat X^{n,m}_0\to \check x_0$ in probability and
$\int_{\R_+^\ell}\abs{u}\hat X^{n,m}_0(u)\,du\to
\int_{\R_+^\ell}\abs{u}\check x_0(u)\,du$ in probability,
  as $n,m\to\infty$\,, and 
\begin{equation*}
  \lim_{K\to\infty}
\limsup_{n,m\to\infty}
\mathbf P(\sum_i\abs{X^{n,m}_0(i)}^2
>Km^2n^\ell)=0\,,
\end{equation*}
then the  $\hat X^{n,m}$ converge in
probability in  $\mathbb D(\R_+,
 \mathbb L^2(\R_+^\ell)\cap
 \mathbb L^1(\R_+^\ell))$   to 
  function $(\check x_t\,, t\ge0)$\,, where 
$\check x_t =(\check x_t(u),\,u\in
\R_+^\ell)$\,, such that  $\check x_t(u)\ge 0$\,, 
$\int_{\R_+^\ell} \check x_t(u)
\,du>0$\,, $\check x_t\in \mathbb L^2(\R_+^\ell)\cap 
\mathbb L^1(\R_+^\ell)$\,, 
   and, for all $t$ and almost all  $u$\,, the measure 
$\hat\lambda_t(du)=\check x_t(u)\,du$ satisfies \eqref{eq:13} with 
$\hat\eta(u)$ and $\hat\epsilon(u)$ replaced with $\check\eta(u)$ and
$\check\epsilon(u)$\,, respectively.
\item In probability, locally uniformly in $t$\,,
\begin{equation*}
        \int_{\R_+^\ell}\abs{u}\hat X^{n,m}_t(u)\,du
\to \int_{\R_+^\ell}\abs{u}\check x_t(u)\,du\,.
\end{equation*}
If, in addition,
\begin{equation*}
\int_{\R_+^\ell}u\,\hat X^{n,m}_0(u)\,du\to
\int_{\R_+^\ell}u\,\check x_0(u)\,du  
\end{equation*}
 in probability,
then
\begin{equation*}
      \int_{\R_+^\ell}u\,\hat X^{n,m}_t(u)\,du
\to \int_{\R_+^\ell}u\,\check x_t(u)\,du
\end{equation*}
  in
probability locally uniformly in $t$\,.
\item If, under the hypotheses of part 1, 
$\check x_0$ is a bounded and Lipschitz--continuous function,
$\check \varphi(u,u')$ 
has Sobolev derivative $D_u\check\varphi(u,u')$
with respect to $u$ for almost all $u'$ such that
$\esssup_{u\in\R_+^\ell}\int_{\R_+^\ell}( \check
  \varphi(u,u')+
  \abs{D_u\check\varphi(u,u')})\,du'<\infty$\,, 
and the functions
  $\hat\phi(u)$ and $\check\epsilon(u)$ are Lipschitz--continuous,
then $\check x_t(u)$ is a bounded and Lipschitz--continuous function
with respect to $u$\,, is locally Lipschitz--continuous with respect
to $t$\,, and \eqref{eq:28} is satisfied for
almost all $t$ and  $u$\,.
  \end{enumerate}
  \end{theorem}
\section{Proof of Theorem \ref{the:lln} }
\label{sec:proof-theorem-}
The proof of Theorem \ref{the:lln}
 proceeds by establishing compactness of $\Lambda^{n,m}$ and
ascertaining the limit point. Techniques of stochastic calculus are used extensively.
Throughout the section, the hypotheses of part 1 of 
Theorem \ref{the:lln} are assumed to hold.
We begin with a lemma on the properties of fission.
Let $b$ denote an upper bound on the  number of offspring in a fission.
\begin{lemma}
\label{le:conserv}
We have that
\begin{equation}
\label{eq:90}   \hat\eta^{n,m}(u,\R_+^\ell)\le b
\end{equation}
and \begin{equation}
\label{eq:90a}\int_{\R_+^\ell}u'\hat\eta^{n,m}(u,du')
=\frac{\lfloor nu\rfloor}{n}\,.
\end{equation}
As a result, 
\begin{equation}
  \label{eq:16}
  \hat\varphi(u,\R_+^\ell)\le b\hat\phi(u)
\end{equation}
and
\begin{equation}
  \label{eq:91}
    \int_{\R_+^\ell}u'\hat\varphi(u,du')=u\hat\phi(u)\,.
\end{equation}
\end{lemma}
\begin{proof}
By the analogue of \eqref{eq:97},
\begin{equation*}
 \sum_{i'} \eta^{n,m}(i,i')=\sum_{i'}\mathbf E \theta^{n,m}_i(i')=
\mathbf E\sum_{i'} \theta^{n,m}_i(i')\,.
\end{equation*}
The latter sum is the total number of pieces, so, it 
does not exceed $b$\,. Similarly,
\begin{equation*}
  \sum_{i'}i' \eta^{n,m}(i,i')=
\mathbf E\sum_{i'}i' \theta_i^{n,m}(i')\,,
\end{equation*}
the latter sum being equal to $i$\,.
Representations \eqref{eq:90} and \eqref{eq:90a} now follow from  
\eqref{eq:32}.
Since $\hat\eta^{n,m}(u,du')=0$ when $\abs{u'}>\abs{u}$\,,
\eqref{eq:16} and \eqref{eq:91} follow from \eqref{eq:7},
\eqref{eq:90} and \eqref{eq:90a}.
\end{proof}
\begin{remark}
Similarly,   $\hat\eta^{n,m}(u,\R_+^\ell)= b$ 
and
$  \hat\varphi(u,\R_+^\ell)= b\hat\phi(u)\,,
$ provided every fission produces exactly $b$ offspring.
If, in addition, $\hat\eta^{n,m}(u,du')\to\hat\eta(u,du')$ weakly, then
these relations carry over to $\hat\eta(u,du')$\,.
\end{remark}
Let $\mathcal{F}^{n,m}_t$ represent the complete $\sigma$--algebra
that is generated by the random variables 
$X_0^{n,m}(i)$\,, $L^{B,n,m,k}_s(i,p,r)$\,,
 $L^{D,n,m,k}_s(i,p,r)$\,,
 $L^{\overline M,n,m,k}_s(i,p,r)$\,, 
 $L^{\overline F,n,m}_s(i,p)$\,, 
 $L^{ E,n,m}_s(i,p)$\,,
$B^{n,m,k}_s(i)$\,, $D^{n,m,k}_s(i)$\,, 
$\overline M^{n,m, k}_s(i)$\,, $M^{n,m, k}_s(i)$\,, and
$F^{n,m}_s(i,i')$\,, where $i\in\mathbb Z_+^\ell\setminus\{0\}$\,,
 $p\in\N$\,, $r\in\N$\,, $k\in\{1,\ldots,\ell\}$\,, and $0\le s\le t$\,, and let $\mathbf
 F^{n,m}=(\mathcal{F}^{n,m}_t\,,t\ge0)$ represent the associated filtration.
Let us adopt the convention that $0/0=0$\,, that the analogues of 
the processes on the
lefthand side of \eqref{eq:6} are equal to zero 
when $i=0$ and define
\begin{align}
  \label{eq:18}
  \begin{split}
                     N_t^{B,n,m,k}(i)=B_t^{n,m,k}(i)-\int_0^tX^{n,m}_s(i)i_k\beta^{n,m,k}(i)\,ds\,,\\
N_t^{D,n,m,k}(i)=D_t^{n,m,k}(i)-\int_0^tX^{n,m}_s(i)i_k\delta^{n,m,k}(i)\,ds\,,
\\
N^{F,n,m}_t(i',i)=F^{n,m}_t(i',i)-\int_0^t
X^{n,m}_s(i')\phi^{n,m}(i')\eta^{n,m}(i',i)
\,ds\,,\\
N^{\overline F,n,m}_t(i)=\overline{F}^{n,m}_t(i)-\int_0^t
X^{n,m}_s(i)\phi^{n,m}(i)\,ds\,,\\
N^{E,n,m}_t(i)=E^{n,m}_t(i)-\int_0^t
X^{n,m}_s(i)X_s^{n,m,\ast}\epsilon^{n,m}(i)\,ds\,,\\
    N_t^{M,n,m,k}(i)=M_t^{n,m,k}(i)-
\int_0^t\sum_{i'\not=i}
X^{n,m}_s(i')i_k'\mu^{n,m,k}(i')
\frac{X^{n,m}_s(i)}{X_s^{n,m,\ast}}\,ds\,,
\\N_t^{\overline{M},n,m,k}(i)=\overline
M_t^{n,m,k}(i)-\int_0^tX^{n,m}_s(i)i_k\mu^{n,m,k}(i)\bl(1-\frac{X^{n,m}_s(i)}{X_s^{n,m,\ast}}\br)
\,ds\,.
  \end{split}
 \end{align}
We note that   the righthand sides  are equal to zero
after the time when $X^{n,m,\ast}_t$ hits zero.
Let 
\begin{equation*}
  \alpha^{n,m}_i(i')=\mathbf E\theta^{n,m}_i(i',1)^2\,.
\end{equation*}
We note that 
\begin{equation}
  \label{eq:104}
  \alpha^{n,m}_i(i')\le b \eta^{n,m}(i,i')\,.
\end{equation}
By  the analogues of \eqref{eq:97} and
\eqref{eq:6}, and by Lemma \ref{le:split}  in
the appendix, the processes on the righthand  sides of \eqref{eq:18} are
locally square integrable martingales, whose
 predictable quadratic variation processes are as follows, see, e.g.,
 Liptser and Shiryayev \cite{lipshir} for the corresponding definitions,
 \begin{align}
   \label{eq:28a}
        \begin{split}
           \langle N^{B,n,m,k}(i)\rangle_t
=\int_0^tX^{n,m}_s(i)i_k\beta^{n,m,k}(i)\,ds\,,\;
\langle
N^{D,n,m,k}(i)\rangle_t=\int_0^tX^{n,m}_s(i)i_k\delta^{n,m,k}(i)\,ds\,,
\\
\langle N^{E,n,m}(i)\rangle_t=\int_0^t
X^{n,m}_s(i)X_s^{n,m,\ast}\epsilon^{n,m}(i)\,ds\,,\;
\langle N^{\overline F,n,m}(i)\rangle_t=\int_0^t
X^{n,m}_s(i)\phi^{n,m}(i)\,ds\,,
\\  \langle N^{F,n,m}(i',i)\rangle_t=
    \int_0^t
X^{n,m}_s(i')\phi^{n,m}(i')\alpha^{n,m}_{i'}(i)
\,ds\,,\\
\langle N^{M,n,m,k}(i)\rangle_t=\int_0^t\sum_{i'\not=i}
X^{n,m}_s(i')i_k'\mu^{n,m,k}(i')
\frac{X^{n,m}_s(i)}{X_s^{n,m,\ast}}\,ds\,,
\\
\langle
N^{\overline{M},n,m,k}(i)\rangle_t=\int_0^tX^{n,m}_s(i)i_k\mu^{n,m,k}(i)
\bl(1-
\frac{X^{n,m}_s(i)}{X_s^{n,m,\ast}}\br)
\,ds\,.
  \end{split}
\end{align}
The nonzero predictable covariance processes are
\begin{align}\label{eq:66}
  \begin{split}
          \langle N^{M,n,m,k}(i),N^{\overline M,n,m,k}(i')\rangle_t=&
(1-\mathbf 1_i(i'))\int_0^tX^{n,m}_s(i')
i'_k\mu^{n,m,k}(i')
\frac{X^{n,m}_s(i)}{X_s^{n,m,\ast}}\,ds\,,
\\  \langle  N^{F,n,m}(i,i'),N^{\overline F,n,m}(i)\rangle_t
=&\int_0^tX^{n,m}_s(i)\phi^{n,m}(i)\eta^{n,m}(i,i')
\,ds\,,\\
\langle
N^{F,n,m}(i,i'),N^{F,n,m}(i,j')\rangle_t
=&\int_0^tX^{n,m}_s(i)\phi^{n,m}(i)
\mathbf E\theta^{n,m}_i(i',1)\theta^{n,m}_i(j',1)
\,ds\,.
  \end{split}
\end{align}

Let
\begin{equation}
  \label{eq:19}
  R^{n,m}_t=\frac{1}{m}\,X_t^{n,m,\ast}=\Lambda^{n,m}_t(\R_+^\ell)\,,
\end{equation}
so,
$R^{n,m}_0\to R_0=\hat\lambda_0(\R_+^\ell)>0$ 
in probability, as
$n,m\to\infty$\,.
The processes  $(R^{n,m}_t\,,t\ge0)$ are random
 elements of $\mathbb D(\R_+,\R)$\,.
Let us recall that a sequence of  stochastic processes with
trajectories in a Skorohod space is said to be 
$C$--tight if it is tight
for convergence in distribution in the Skorohod space
and the limit points are laws of continuous path processes, see, e.g.,
Jacod and Shiryaev \cite{jacshir}.
\begin{lemma}
  \label{le:numberofgroups}
 The sequence of processes $(R^{n,m}_t\,,t\ge0)$ is $C$--tight
and, given $t>0$\,, there exists $\rho>0$ such that
$\mathbf P(\inf_{s\le t}R^{n,m}_s>\rho)\to1$\,, as $n,m\to\infty$\,.
\end{lemma}
\begin{proof}
  By the analogue of \eqref{eq:1},
\begin{equation}
  \label{eq:27}
      \Delta X_t^{n,m,\ast}
=\sum_{i}
\sum_{i'}\Delta F^{n,m}_t(i',i)
-\sum_{k=1}^\ell\Delta D^{n,m,k}_t(e_k)-\sum_{i}\Delta \overline{F}^{n,m}_t(i)
\\ -\sum_{i}\Delta
E_t^{n,m}(i)\,.
\end{equation}
 Since $\sum_{i}
\Delta F^{n,m}_t(i',i)\le b \Delta \overline F^{n,m}_t(i')$\,,
\begin{equation*}
      \Delta X_t^{n,m,\ast}
\le b\sum_{i}\Delta \overline{F}^{n,m}_t(i)
\\ -\sum_{i}\Delta
E_t^{n,m}(i)\,.
\end{equation*}
Hence,
\begin{equation}
  \label{eq:23}
  \Delta R^{n,m}_t\le\Delta\overline R^{n,m}_t \,,
\end{equation}
where
\begin{equation}
  \label{eq:21}
  \overline R^{n,m}_t=R^{n,m}_0+
\frac{b}{m}\,\sum_{i} \overline{F}^{n,m}_t(i)
-\,\frac{1}{m}\,\sum_{i} E_t^{n,m}(i)\,.
\end{equation}
Let $
  N^{\overline R,n,m}_t(i)=(bN^{\overline F,n,m}_t(i)-N^{E,n,m}_t(i))/m
$ and $  N^{\overline R,n,m}_t=\sum_{i}  N^{\overline R,n,m}_t(i)$ so that by  \eqref{eq:18},
\begin{equation}
  \label{eq:22}
  \overline R^{n,m}_t=R^{n,m}_0+\int_0^t\sum_{i}\,
\frac{1}{m}\, X_s^{n,m}(i)\bl(b\phi^{n,m}(i)-
X^{n,m,\ast}_s\epsilon^{n,m}(i)\br)\,ds+
\,N^{\overline R,n,m}_t\,.
\end{equation}
 Since the processes $N^{\overline F,n,m}(i)=(N^{\overline F,n,m}_t(i)\,,t\ge0)$ and
$N^{E,n,m}(i)=(N^{E,n,m}_t(i)\,,t\ge0)$ are locally square integrable
martingales with  disjoint jumps, it
follows by  \eqref{eq:28a} that the process
$N^{\overline R,n,m}=(N^{\overline R,n,m}_t\,,t\ge0)$ is a locally square integrable martingale with
the predictable quadratic variation process
\begin{equation}
  \label{eq:25}
  \langle N^{\overline R,n,m}\rangle _t=
\int_0^t\sum_{i}\frac{1}{m^2}
X^{n,m}_s(i)(b^2\phi^{n,m}(i)+X^{n,m,\ast}_s
\epsilon^{n,m}(i))\,ds\,.
\end{equation}
By \eqref{eq:18},  \eqref{eq:21}, \eqref{eq:22}, and the It\^o formula for
semimartingales, see, e.g., Theorem 1 on p.118 in
 Liptser and Shiryayev \cite{lipshir}, on taking into account that 
the processes $(\overline F^{n,m}_t\,,t\ge0)$ and 
$(E^{n,m}_t\,,t\ge0)$ have unit jumps,
\begin{multline*}
  (\overline R^{n,m}_t)^2
=(  R^{n,m}_0)^2+\int_0^t2 \overline R^{n,m}_{s-}d
 \overline R^{n,m}_s+\frac{1}{m^2}
\sum_{s\le t}(b^2\sum_{i}(\Delta
 \overline{F}^{n,m}_s(i))^2+
\sum_{i}(\Delta E_s^{n,m}(i))^2)\\
=( R^{n,m}_0)^2+\int_0^t2 \overline R^{n,m}_{s}
\sum_{i}\,
\frac{1}{m}\, X_s^{n,m}(i)\bl(b\phi^{n,m}(i)-
R^{n,m}_sm\epsilon^{n,m}(i)\br)\,ds\\+
\int_0^t2\overline R^{n,m}_{s-}dN^{R,n,m}_s+\frac{1}{m^2}
\sum_{i}(b^2 \overline{F}^{n,m}_t(i)+
 E_t^{n,m}(i))\\
\le( R^{n,m}_0)^2+\int_0^t2\overline R^{n,m}_{s}
\sum_{i}\,
\frac{b}{m}\, X_s^{n,m}(i)\phi^{n,m}(i)\,ds
\\+
\int_0^t2\overline R^{n,m}_{s-}dN^{R,n,m}_s+\frac{1}{m^2}
\sum_{i}(b^2 \overline{F}^{n,m}_t(i)+
 E_t^{n,m}(i))\,.
\end{multline*}
Hence, on recalling  \eqref{eq:18}, \eqref{eq:19}, and \eqref{eq:23},
by $m\epsilon^{n,m}(i)$ and 
$\phi^{n,m}(i)$ being bounded, there exists $K_0>0$ such that, for all $t>0$\,,
\begin{multline}
  \label{eq:38}
    (\overline R^{n,m}_t)^2\le ( R^{n,m}_0)^2
+\frac{K_0}{m}\,t
+K_0\int_0^t (\overline R^{n,m}_s)^2\,ds
+
\int_0^t2\overline R^{n,m}_{s-}dN^{\overline R,n,m}_s\\+\frac{1}{m^2}
\sum_{i}(N^{\overline F,n,m}_t(i)+N^{E,n,m}_t(i))\,.
\end{multline}
Let, for $K_1>0$\,,
\begin{equation*}
  \tau_{K_1}=\inf\{s\ge0:\,\overline R^{n,m}_s>K_1\}\,.
\end{equation*}
Then, by $R^{n,m}_0$ being $\mathcal{F}^{n,m}_0$--measurable and by
$N^{\overline R,n,m}$\,, $N^{\overline F,n,m}(i)$ and $N^{E,n,m}(i)$ being
locally square integrable
martingales, whose predictable quadratic variation processes are bounded for
$t\le\tau_{K_1}$ by \eqref{eq:28a} and \eqref{eq:25}, so that the local martingale on the righthand side
of \eqref{eq:38} stopped at $\tau_{K_1}$ is a martingale,  we have that
\begin{equation*}
  \mathbf E
\bl(\int_0^{t\wedge \tau_{K_1}}2\overline R^{n,m}_{s-}dN^{\overline R,n,m}_s\\+\frac{1}{m^2}
\sum_{i}(N^{\overline F,n,m}_{t\wedge \tau_{K_1}}(i)
+N^{E,n,m}_{t\wedge \tau_{K_1}}(i))\br)
=0\,.
\end{equation*}
 By \eqref{eq:38} and Gronwall's inequality, for  $K_2>0$\,,
  \begin{equation*}
    \mathbf E  (\overline R^{n,m}_{t\wedge \tau_{K_1}})^2\mathbf 1_{[0,K_2]}(R^{n,m}_0)
\le (K_2^2
+\frac{K_0}{m}\,t)
e^{K_0t}\,.
  \end{equation*}
Letting $K_1\to\infty$ implies, by Fatou's lemma, that
\begin{equation}
  \label{eq:14}
      \mathbf E  (\overline R^{n,m}_{t})^2\mathbf 1_{[0,K_2]}(R^{n,m}_0)\le
( K_2^2
+\frac{K_0}{m}\,t)
e^{K_0t}\,.
\end{equation}
By \eqref{eq:25}, \eqref{eq:14}, the $\phi^{n,m}(i)$ and
$m\epsilon^{n,m}(i)$ being bounded,   for $\gamma>0$\,,
\begin{equation}
  \label{eq:35}
    \lim_{n,m\to\infty}
\mathbf P(\langle N^{\overline R,n,m}\rangle _t>\gamma)=0\,,
\end{equation}
so, by
the
Lenglart--Rebolledo inequality, see, e.g.,
Theorem 3 on p.66 in Liptser and Shiryayev
\cite{lipshir},  in probability, 
\begin{equation*}
    \lim_{m,n\to\infty}\sup_{s\le
    t}\abs{N^{\overline R,n,m}_s}=0\,.
\end{equation*}
By  \eqref{eq:19}, \eqref{eq:22}, and Gronwall's inequality, for some $K'>0$\,,
\begin{equation*}
\overline  R^{n,m}_t\le (R^{n,m}_0+\sup_{s\le t}\abs{N^{\overline R,n,m}_s})e^{K't}\,.
\end{equation*}
It follows that
\begin{equation}
  \label{eq:30}
   \lim_{K\to\infty} \limsup_{n,m\to\infty}
\mathbf P(\sup_{s\le t}\overline R^{n,m}_s>K)=0\,.
\end{equation}
By  \eqref{eq:18} and \eqref{eq:27}, 
\begin{multline}
  \label{eq:24}
  R^{n,m}_t=
R^{n,m}_0
+\int_0^t\sum_{i}\,
\frac{X_s^{n,m}(i)}{m}\, \bl(\sum_{i'}\eta^{n,m}(i,i')\phi^{n,m}(i)
-\phi^{n,m}(i)-
X^{n,m,\ast}_s\epsilon^{n,m}(i)\br)\,ds\\
-\sum_{k=1}^\ell \int_0^t\frac{X^{n,m}_s(e_k)}{m}
\,\delta^{n,m,k}(e_k)\,ds
-\sum_{k=1}^\ell \frac{N^{D,n,m,k}_t(e_k)}{m}\,+
\frac{N^{R,n,m}_t}{m}\,,
\end{multline}
where $(N^{R,n,m}_t\,,t\ge0)$ is a locally square integrable
martingale
with the predictable quadratic variation process
\begin{equation*}
    \langle N^{R,n,m}\rangle_t=
\sum_i\int_0^tX^{n,m}_s(i)
\bl( \mathbf E\bl(\sum_{i'}\theta^{n,m}_i(i',1)-1\br)^2\phi^{n,m}(i)
+X^{n,m,\ast}_s\epsilon^{n,m}(i)\br)\,ds\,.
\end{equation*}
Since $\sum_{i'}\theta_i^{n,m}(i',1)\le b$\,, by \eqref{eq:25} and
\eqref{eq:35},  we have that
$\langle N^{R,n,m}\rangle_t/m^2\to$ in probability, as $n,m\to\infty$\,,
so, in probability,
\begin{equation}
  \label{eq:102}
  \lim_{n,m\to\infty}\frac{1}{m}\,\sup_{s\le t}\abs{N^{R,n,m}}=0\,.
\end{equation}
Hence,  for  arbitrary $K_3>0$\,,
\begin{multline}
  \label{eq:99}
      R^{n,m}_t-\frac{N^{R,n,m}_t}{m}+
\sum_{k=1}^\ell \frac{N^{D,n,m,k}_t(e_k)}{m}
=
R^{n,m}_0-K_3\int_0^t(R^{n,m}_s-
\frac{N^{R,n,m}_s}{m}\\+\sum_{k=1}^\ell \frac{N^{D,n,m,k}_s(e_k)}{m})
\,ds+
\int_0^t\bl(K_3R^{n,m}_s-\sum_{k=1}^\ell \frac{X^{n,m}_s(e_k)}{m}
\,\delta^{n,m,k}(e_k)\\
+\sum_{i}\,
\frac{X_s^{n,m}(i)}{m}\, (\sum_{i'}\eta^{n,m}(i,i')\phi^{n,m}(i)-\phi^{n,m}(i)-
X^{n,m,\ast}_s\epsilon^{n,m}(i))-K_3\frac{N^{R,n,m}_s}{m}\\
+K_3\sum_{k=1}^\ell \frac{N^{D,n,m,k}_s(e_k)}{m}\br)\,ds\,.
\end{multline}
By \eqref{eq:99} and the fact that 
$\sum_{i'}\eta^{n,m}(i,i')\le b$ according to Lemma \ref{le:conserv},
 if $\sup_{s\le t}R^{n,m}_s\\\le K$\, then, solving for the lefthand
side of \eqref{eq:99} and  picking  $K_3$  great
enough so that
\begin{multline*}
 K_3R^{n,m}_s-\sum_{k=1}^\ell \frac{X^{n,m}_s(e_k)}{m}
\,\delta^{n,m,k}(e_k)+\sum_{i}\,
\frac{X_s^{n,m}(i)}{m}\, (\sum_{i'}\eta^{n,m}(i,i')\phi^{n,m}(i)\\-\phi^{n,m}(i)-
X^{n,m,\ast}_s\epsilon^{n,m}(i))\ge0 
\end{multline*}
  when $0\le s\le t$\,, 
\begin{multline}
  \label{eq:88}
    R^{n,m}_t-\frac{N^{R,n,m}_t}{m}+\sum_{k=1}^\ell \frac{N^{D,n,m,k}_t(e_k)}{m}
=
e^{-K_3t}R^{n,m}_0+
e^{-K_3t}\int_0^te^{K_3s}\bl(K_3R^{n,m}_s\\
-\sum_{k=1}^\ell \frac{X^{n,m}_s(e_k)}{m}
\,\delta^{n,m,k}(e_k)+\sum_{i}\,
\frac{X_s^{n,m}(i)}{m}\, (\sum_{i'}\eta^{n,m}(i,i')\phi^{n,m}(i)\\-\phi^{n,m}(i)-
X^{n,m,\ast}_s\epsilon^{n,m}(i))
-K_3\frac{N^{R,n,m}_s}{m}+K_3\sum_{k=1}^\ell \frac{N^{D,n,m,k}_s(e_k)}{m}
\br)\,ds\\
\ge 
e^{-K_3t}R^{n,m}_0-\sup_{s\le t}\frac{\abs{N^{R,n,m}_s}}{m}
-\sum_{k=1}^\ell \sup_{s\le t}\frac{\abs{N^{D,n,m,k}_s(e_k)}}{m}\,.
\end{multline}
 By  \eqref{eq:28a}, \eqref{eq:19},
 \eqref{eq:30}, the $\delta^{n,m,k}(e_k)$ being
 bounded   and the
Lenglart--Rebolledo inequality, in probability, for $k=1,\ldots,\ell$\,,
\begin{equation*}
  \lim_{m,n\to\infty}\sup_{s\le
    t}\frac{1}{m}\,\abs{N^{D,n,m,k}_t(e_k)}=0\,,
\end{equation*}
which implies, by \eqref{eq:102},  \eqref{eq:88}, and  $\hat\lambda_0(\R_+^\ell)$
being positive, that the $R^{n,m}_t$ are  locally
uniformly asymptotically separated away from zero in probability.
In addition, by \eqref{eq:24} and the 
$m\epsilon^{n,m}(i)$ being bounded,
\begin{equation*}
  \lim_{\sigma\to0}\limsup_{n,m\to\infty}
\mathbf P\bl(\sup_{s,s'\in[0,t]:\,\abs{s-s'}\le \sigma}\abs{
    R^{n,m}_{s}- R^{n,m}_{s'}}>\gamma\br)=0\,,
\end{equation*}
so, $ R^{n,m}$ is $C$--tight.

\end{proof}
Let us introduce  
$\hat\beta^{n,m}(u)=(\hat\beta^{n,m,1}(u),\ldots,\hat\beta^{n,m,\ell}(u))$ and
$\hat\delta^{n,m}(u)=(\hat\delta^{n,m,1}(u),\ldots,\linebreak\hat\delta^{n,m,\ell}(u))$\,, and let $\cdot$
denote the inner product in $\R^\ell$\,.
\begin{lemma}
  \label{le:yule}
The sequence of 
processes $\bl(\int_{\R_+^\ell} \abs{u}\Lambda^{n,m}_t(du)\,, t\ge0\br)$ is $C$--tight.
\end{lemma}
\begin{proof}
  By the analogue of \eqref{eq:1},
\begin{align*}
  \sum_{i}i_k\Delta X^{n,m}_t(i)=
\sum_{i}\Delta B^{n,m,k}_t(i)-\sum_{i}\Delta D^{n,m,k}_t(i)-
\sum_{i}i_k\Delta E^{n,m}_t(i)\,,
\end{align*}
so,
\begin{multline}
  \label{eq:12}
    \sum_{i}\sum_{k=1}^\ell i_k X^{n,m}_t(i)= \sum_{i}
\sum_{k=1}^\ell i_k X^{n,m}_0(i)+
\sum_{i}\sum_{k=1}^\ell B^{n,m,k}_t(i)\\
-\sum_{i}\sum_{k=1}^\ell  D^{n,m,k}_t(i)
-\sum_{i}\abs{i} E^{n,m}_t(i)\,.
\end{multline}
Therefore, on recalling \eqref{eq:2} and \eqref{eq:18},
\begin{multline}
  \label{eq:26}
  \int_{\R_+^\ell}\abs{u}\Lambda^{n,m}_t(du)\le
\int_{\R_+^\ell}\abs{u}\Lambda^{n,m}_0(du)
+\int_0^t
\int_{\R_+^\ell}u\cdot\hat\beta^{n,m}(u)
\Lambda^{n,m}_s(du)\,ds\\+\frac{1}{mn}\, N^{B,n,m}_t\,,
\end{multline}
where
\begin{equation*}
  N^{B,n,m}_t=
\sum_{i}\sum_{k=1}^\ell N^{B,n,m,k}_t(i)\,.
\end{equation*}
By \eqref{eq:28a}, 
the process $N^{B,n,m}=(N^{B,n,m}_t\,,t\ge0)$ is a locally square
integrable  martingale
with the predictable quadratic variation process
\begin{multline}
  \label{eq:10}
\langle
N^{B,n,m}\rangle_t=\int_0^t\sum_{i}\sum_{k=1}^\ell i_k\beta^{n,m,k}(i)
X^{n,m}_s(i)\,ds=
nm\int_0^t\int_{\R_+^\ell}u\cdot\hat\beta^{n,m}(u)\Lambda^{n,m}_s(du)\,ds
\,. 
\end{multline}
Let $\Theta^{n,m}_K$ represent the event that 
$\int_{\R_+^\ell}\abs{u}\Lambda^{n,m}_0(du)\le K$\,, where $K>0$\,,
and let, for $K_1>0$\,,
\begin{equation*}
    \tau_{K_1}=\inf\{t\ge0:\,
\int_0^t\int_{\R_+^\ell}u\cdot \hat\beta^{n,m}(u)\Lambda^{n,m}_s(du)\,ds>K_1\} \,.
\end{equation*}
The process $(N^{B,n,m}_{t\wedge \tau_{K_1}}\,,t\ge0)$ being a
martingale implies
  by \eqref{eq:26} that\\
$  \mathbf E\mathbf 1_{\Theta^{n,m}(K)}
\int_{\R_+^\ell}\abs{u}\Lambda^{n,m}_{t\wedge \tau_{K_1}}(du)$
is finite, so
 by the $\hat\beta^{n,m,k}(u)$ being bounded, 
provided $K$ is great enough,
\begin{equation*}
  \mathbf E\mathbf 1_{\Theta^{n,m}(K)}
  \int_{\R_+^\ell}\abs{u}\Lambda^{n,m}_{t\wedge \tau_{K_1}}(du)\le
  K
+K\int_0^{t}
\mathbf E\mathbf 1_{\Theta^{n,m}(K)}\int_{\R_+^\ell}
\abs{u}
\Lambda^{n,m}_{s\wedge \tau_{K_1}}(du)\,ds\,.
\end{equation*}
By Gronwall's inequality and Fatou's lemma,
\begin{equation*}
  \mathbf E\mathbf 1_{\Theta^{n,m}(K)}
  \int_{\R_+^\ell}\abs{u}\Lambda^{n,m}_{t}(du)
\le Ke^{Kt}\,.
\end{equation*}
Since, for $K_2>0$\,,
\begin{equation*}
  \mathbf P(\int_{\R_+^\ell}\abs{u}\Lambda^{n,m}_{t}(du)>K_2)\le
1-\mathbf P(\Theta^{n,m}_K)+\frac{Ke^{Kt}}{K_2}\,,
\end{equation*}
we have that
\begin{equation*}
\lim_{K_2\to\infty}\limsup_{n,m\to\infty}
  \mathbf P(\int_{\R_+^\ell}\abs{u}\Lambda^{n,m}_{t}(du)>K_2)\le
\limsup_{n,m\to\infty}(1-\mathbf P(\Theta^{n,m}_K))\,,
\end{equation*}
which implies that the lefthand side equals zero by $K$ being
arbitrary and by the fact that
$\int_{\R^\ell_+}\abs{u}\Lambda^{n,m}_0(du)\to
\int_{\R^\ell_+}\abs{u}\hat\lambda_0(du)$ in probability.
By \eqref{eq:10}, for $\gamma>0$\,,
\begin{equation*}
 \lim_{n,m\to\infty}\mathbf P\bl( \frac{1}{m^2n^2}\langle
N^{B,n,m}\rangle_t>\gamma\br)=0\,,
\end{equation*}
so, by the Lenglart--Rebolledo inequality,
\begin{equation}
  \label{eq:62}
 \lim_{n,m\to\infty}\mathbf P\bl( \frac{1}{mn}\sup_{s\le t}\abs{
N^{B,n,m}_s}>\gamma\br)=0\,.
\end{equation}
By \eqref{eq:26} and Gronwall's inequality, for $K$ great enough,
\begin{equation*}
  \sup_{s\le t}\int_{\R_+^\ell}\abs{u}\Lambda^{n,m}_s(du)\le
\bl(\int_{\R_+^\ell}\abs{u}\Lambda^{n,m}_0(du)+
\sup_{s\le t}\frac{1}{mn}\,\abs{ N^{B,n,m}_t}\br)e^{Kt}\,.
\end{equation*}
By \eqref{eq:62}, 
 \begin{equation}
  \label{eq:20}
    \lim_{K_2\to\infty}\limsup_{n,m\to\infty}
  \mathbf P\bl(\sup_{s\le t}\int_{\R_+^\ell}\abs{u}\Lambda^{n,m}_s(du)>K_2\br)=0\,.
\end{equation}
By \eqref{eq:18} and \eqref{eq:12}, in analogy with \eqref{eq:26},  for $s\le t$\,,
\begin{multline*}
\abs{ \int_{\R_+^\ell}\abs{u}  \Lambda^{n,m}_{t}(du)
-\int_{\R_+^\ell}\abs{u}\Lambda^{n,m}_s(du)}\\
\le
\int_s^t
\int_{\R_+^\ell}\bl(
 u\cdot(\hat\beta^{n,m}(u)+\hat\delta^{n,m}(u))
+\abs{u}\hat\epsilon^{n,m}(u) R^{n,m}_{\tilde s}\br)\Lambda^{n,m}_{\tilde s}(du)\,d\tilde s\\+\frac{1}{mn}\,
\abs{ (N^{B,n,m}_{t}-N^{B,n,m}_s)+ 
(N^{D,n,m}_{t}-N^{D,n,m}_s)+ 
(N^{E,n,m}_{t}-N^{E,n,m}_s)}\\
\le
(t-s)
\sup_{\tilde s\in[s,t]}\int_{\R_+^\ell}\bl(
 u\cdot(\hat\beta^{n,m}(u)+\hat\delta^{n,m}(u))
+\abs{u}\hat\epsilon^{n,m}(u) R^{n,m}_{\tilde s}\br)\Lambda^{n,m}_{\tilde s}(du)\\+\frac{1}{mn}\,
\abs{ (N^{B,n,m}_{t}-N^{B,n,m}_s)+ 
(N^{D,n,m}_{t}-N^{D,n,m}_s)+ 
(N^{E,n,m}_{t}-N^{E,n,m}_s)}\,,
\end{multline*}
where
\begin{equation*}
  N^{D,n,m}_{\tilde t}=
\sum_{i}\sum_{k=1}^\ell N^{D,n,m,k}_{\tilde t}(i)\,,\quad
N^{E,n,m}_{\tilde t}=\sum_i N^{E,n,m}_{\tilde t}(i)\,.
\end{equation*}
Similarly to \eqref{eq:62}, on recalling \eqref{eq:28a},
  \begin{equation*}
  \lim_{n,m\to\infty}\mathbf P\bl(\sup_{t\le L}
\frac{\abs{N^{D,n,m}_t}+\abs{N^{E,n,m}_t}}{mn}>\gamma\br)=0\,.
\end{equation*}

Hence, invoking \eqref{eq:20} once again and Lemma \ref{le:numberofgroups},
\begin{equation*}
    \lim_{\chi\to0}\limsup_{n,m\to\infty}
\mathbf P\bl(
\sup_{s,t\le L,\,\abs{t-s}<\chi}\abs{
  \int_{\R_+^\ell}\abs{u}(\Lambda^{n,m}_{t}(du)-\Lambda^{n,m}_s(du))}>\gamma\br)=0\,.
\end{equation*}
\end{proof}
We develop more semimartingale decompositions.
 Rearranging in the analogue of \eqref{eq:1}, accounting for
\eqref{eq:6} and \eqref{eq:18} and 
assuming that $X^{n,m}_s(i-e_k)=0$ when  $i_k=0$\,,
yields, for $i\in\mathbb Z_+^\ell$\,,
\begin{multline}
  \label{eq:52a}  X^{n,m}_t(i)-X^{n,m}_0(i)=
\int_0^t\Bl(\sum_{k=1}^\ell \bl(-X^{n,m}_s(i)i_k(\beta^{n,m,k}(i)
+\delta^{n,m,k}(i))
\\+X^{n,m}_s(i-e_k)(i_k-1)\beta^{n,m,k}(i-e_k)+
X^{n,m}_s(i+e_k)(i_k+1)\delta^{n,m,k}(i+e_k)
\\-X^{n,m}_s(i)i_k\mu^{n,m,k}(i)+X^{n,m}_s(i+e_k)(i_k+1)\mu^{n,m,k}(i+e_k)\\
-\sum_{i'\not=i}
X^{n,m}_s(i')i_k'\mu^{n,m,k}(i')
\frac{X^{n,m}_s(i)}{X_s^{n,m,\ast}}
+\sum_{i'\not=i-e_k}
X^{n,m}_s(i')i_k'\mu^{n,m,k}(i')
\frac{X^{n,m}_s(i-e_k)}{X_s^{n,m,\ast}}\br)
\\
-X^{n,m}_s(i)\phi^{n,m}(i)+\sum_{i'}X^{n,m}_s(i')\phi^{n,m}(i')\eta^{n,m}(i',i)
\\
-
X^{n,m}_s(i)X_s^{n,m,\ast}
\epsilon^{n,m}(i)\Br)\,ds+N_t^{n,m}(i)\,,
\end{multline}
where
\begin{align}
  \label{eq:57a}
  \begin{split}
     N_t^{n,m}(i)=
     \sum_{k=1}^\ell \bl( - N_t^{B,n,m,k}(i)- N^{D,n,m,k}_t(i)
+ N^{B,n,m,k}_{t}(i-e_k)
+
N^{D,n,m,k}_t(i+e_k)\\
- N^{\overline M,n,m,k}_t(i)
- N^{M,n,m,k}_t(i)
+ N^{\overline M,n,m,k}_t(i+e_k)
+ N^{M,n,m,k}_{t}(i-e_k)
\br)\\
- N^{\overline{F},n,m}_t(i)+\sum_{i'} N^{F,n,m}_t(i',i)- N^{E,n,m}_t(i)\,.
  \end{split}
\end{align}

Hence,  the predictable quadratic variation process of
$N^{n,m}(i)=(N_t^{n,m}(i)\,,t\ge0)$ is as follows
\begin{multline}
  \label{eq:75a}
      \langle N^{n,m}(i)\rangle_t=
\sum_{k=1}^{\ell}\bl(
\langle N^{B,n,m,k}(i)\rangle_t
+
 \langle N^{D,n,m,k}(i)\rangle_t+\langle N^{B,n,m,k}(i-e_k)\rangle_t
\\+
 \langle N^{D,n,m,k}(i+e_k)\rangle_t
+
 \langle N^{\overline M,n,m,k}(i)\rangle_t+
\langle N^{M,n,m,k}(i)\rangle_t
+
 \langle N^{\overline M,n,m,k}(i+e_k)\rangle_t
\\
+
\langle N^{M,n,m,k}(i-e_k)\rangle_t
+\langle N^{\overline{F},n,m}(i)\rangle_t
+
\sum_{i'}
 \langle N^{F,n,m}(i',i)\rangle_t+
 \langle N^{E,n,m}(i)\rangle_t\\
- \langle N^{\overline M,n,m,k}(i),N^{ M,n,m,k}(i-e_k)\rangle_t
- \langle N^{ M,n,m,k}(i),N^{\overline M,n,m,k}(i+e_k)\rangle_t\\
+ \langle N^{\overline M,n,m,k}(i+e_k),N^{ M,n,m,k}(i-e_k)\rangle_t
\br)
\,,
\end{multline}
where $\langle N^{\overline M,n,m,k}(i),N^{ M,n,m,k}(i-e_k)\rangle_t=
\langle N^{\overline M,n,m,k}(i+e_k),N^{ M,n,m,k}(i-e_k)\rangle_t=0$ when $i_k=0$\,.

\begin{lemma}
  \label{le:dyra}
For all $L>0$\,,
\begin{equation*}\lim_{K\to\infty}
\limsup_{n,m\to\infty}
\mathbf P(\sup_{t\in[0,L]}\sum_i\abs{ X^{n,m}_t(i)}^2
>Km^2)=0\,.
\end{equation*}
\end{lemma}
\begin{proof}
  On writing   \eqref{eq:52a} as
\begin{equation*}
  X^{n,m}_t(i)=X^{n,m}_0(i)+A^{n,m}_t(i)+N^{n,m}_t(i)\,,
\end{equation*}
 we have that
\begin{multline}
  \label{eq:29'}
  X^{n,m}_t(i)^2
=X^{n,m}_0(i)^2+
2\sum_{0<
    s\le t}X^{n,m}_{s-}(i)\Delta X^{n,m}_{s}(i)
+\sum_{0<
    s\le t}(\Delta X^{n,m}_{s}(i))^2\\
\\= X^{n,m}_0(i)^2+
2\int_0^tX^{n,m}_{s}(i)\,d A^{n,m}_{s}(i)
+2\int_0^tX^{n,m}_{s-}(i)
d N^{n,m}_{s}(i)
+\sum_{0<
    s\le t}(\Delta X^{n,m}_{s}(i))^2\,.
\end{multline}
As a consequence, on recalling that $\abs{\Delta X^{n,m}_s(i)}\le
b$\,,
\begin{multline}
  \label{eq:77}
    X^{n,m}_t(i)^2
\le X^{n,m}_0(i)^2+
2\int_0^tX^{n,m}_{s}(i)\,d A^{n,m}_{s}(i)
+2\int_0^tX^{n,m}_{s-}(i)
d N^{n,m}_{s}(i)\\
+b\sum_{0<
    s\le t}\abs{\Delta X^{n,m}_{s}(i)}\,.
\end{multline}
By  \eqref{eq:52a},
\begin{multline*}
  \int_0^tX^{n,m}_{s}(i)\,d A^{n,m}_{s}(i)=
\int_0^tX^{n,m}_{s}(i)\,
\Bl(\sum_{k=1}^\ell\bl(-X^{n,m}_s(i)i_k(\beta^{n,m,k}(i)
+\delta^{n,m,k}(i))\\
+X^{n,m}_s(i- e_k)(i_k-1)\beta^{n,m,k}(i- e_k)
+
X^{n,m}_s(i+ e_k)(i_k+1)\delta^{n,m,k}(i+ e_k)
\\
-X^{n,m}_s(i)i_k\mu^{n,m,k}(i)
+X^{n,m}_s(i+ e_k)(i_k+1)\mu^{n,m,k}(i+ e_k)\\
+\sum_{i'\not=i- e_k}
X^{n,m}_s(i')i_k'\mu^{n,m,k}(i')
\frac{X^{n,m}_s(i- e_k)}{X^{n,m,\ast}_s}
-\sum_{i'\not=i}
X^{n,m}_s(i')i'_k\mu^{n,m,k}(i')
\frac{X^{n,m}_s(i)}{X^{n,m,\ast}_s}\br)\\
+\sum_{i'}X^{n,m}_s(i')\phi^{n,m}(i')\eta^{n,m}(i',i)
-X^{n,m}_s(i)\phi^{n,m}(i)-
X^{n,m}_s(i)X^{n,m,\ast}_s\epsilon^{n,m}(i)\Br)\,ds\,.
\end{multline*}
Since the functions
$i_k(\beta^{n,m,k}(i)+\delta^{n,m,k}(i)+\mu^{n,m,k}(i))$\,,
$\phi^{n,m}(i)$ and $\eta^{n,m}(i',i)$ 
are
bounded, there exist $K_1>0$ and $K_2>0$ such that
\begin{multline*}
 \abs{ \int_0^tX^{n,m}_{s}(i)\,d A^{n,m}_{s}(i)}\le K_1\int_0^t
X^{n,m}_{s}(i)\bl(X^{n,m}_{s}(i)
+\sum_{k=1}^\ell\bl(X^{n,m}_{s}(i-e_k)
+X^{n,m}_{s}(i+e_k)\br)\\
+\sum_{i'}X^{n,m}_{s}(i')\phi^{n,m}(i')\eta^{n,m}(i',i)
+X^{n,m}_{s}(i)X^{n,m,\ast}_{s}\epsilon^{n,m}(i)\br)\,ds\\
\le K_2\int_0^t\bl(X^{n,m}_{s}(i)^2+\sum_{k=1}^\ell\bl(
X^{n,m}_{s}(i-e_k)^2
+X^{n,m}_{s}(i+e_k)^2\br)\\
+X^{n,m}_{s}(i)X^{n,m,\ast}_{s}
+X^{n,m}_{s}(i)^2X^{n,m,\ast}_{s}\epsilon^{n,m}(i)\br)\,ds\,.
\end{multline*}
On recalling that  $m\epsilon^{n,m}(i)$ is
bounded in $n,m,i$\,, for some $K_3>0$\,,
\begin{equation}
  \label{eq:72}
  \sum_{i}\abs{\int_0^tX^{n,m}_{s}(i)\,d A^{n,m}_{s}(i)}\le
K_3\int_0^t\sum_{i}X^{n,m}_{s}(i)^2(1+R^{n,m}_s)\,ds+K_3m^2
\int_0^t(R^{n,m}_{s})^2\,ds\,.
\end{equation}
By \eqref{eq:5}, \eqref{eq:32}, \eqref{eq:28a},
 \eqref{eq:66},  \eqref{eq:19}, and \eqref{eq:75a},
     for some $K_4>0$\,,
\begin{equation}
  \label{eq:76}
\sum_i  d\langle N^{n,m}(i)\rangle_t
\le
 K_4m(1+( R^{n,m}_t)^2)\,dt\,.
\end{equation}
Let, for $\gamma>0$\,,
\begin{equation}
  \label{eq:73}
  \tau^{n,m}_\gamma=\inf\{t\ge0:\, R^{n,m}_t>
\gamma \}\,.
\end{equation}
Let
$X^{n,m}(i)\circ
 N^{n,m}(i)_t=\int_0^tX^{n,m}_{s-}(i)
d N^{n,m}_{s}(i)$ and $X^{n,m}(i)\circ
 N^{n,m}(i)=\bl(X^{n,m}(i)\circ
 N^{n,m}(i)_t\,,t\ge0\br)$\,. 
The process  $X^{n,m}(i)\circ
 N^{n,m}(i)$ is a locally square integrable martingale with the
 predictable quadratic
variation process 
$(\int_0^tX^{n,m}_{s}(i)^2
d\langle N^{n,m}(i)\rangle_s\,,t\ge0\br)$\,. By \eqref{eq:76},
 $\int_0^{t\wedge\tau^{n,m}_\gamma} X^{n,m}_{s}(i)^2
d\langle N^{n,m}(i)\rangle_s$ is bounded for given $n,m$\,. Therefore,
$\bl(X^{n,m}(i)\circ N^{n,m}_{t\wedge\tau^{n,m}_\gamma}(i)\,, t\ge0\br)$  is a martingale, so
$\mathbf E(X^{n,m}(i)\circ N^{n,m}(i)_{t\wedge\tau^{n,m}_\gamma}|\mathcal{F}^{n,m}_0)=0$\,.
By \eqref{eq:77} and \eqref{eq:72}, 
for $R>0$\,,
introducing the event $\Gamma^{n,m}=\{\sum_{i'}X_0^{n,m}(i')^2\le
Rm^2\}$\,, 
\begin{multline*}
   \mathbf E\sum_{i}X^{n,m}_{t\wedge\tau^{n,m}_\gamma}(i)^2
\mathbf 1_{\Gamma^{n,m}}
\le Rm^2+
2K_3(1+\gamma)\int_0^t\mathbf E\sum_{i}X^{n,m}_{s\wedge
  \tau^{n,m}_\gamma}(i)^2\mathbf 1_{\Gamma^{n,m}}\,ds\,
\\
+2K_3
m^2\gamma^2t
+b\mathbf E\sum_{0<
    s\le t}\abs{\Delta X^{n,m}_{s\wedge\tau^{n,m}_\gamma}(i)}\,.
\end{multline*}
By the analogue of \eqref{eq:1}, 
\begin{multline*}
  \mathbf E\sum_{0<
    s\le t}\abs{\Delta X^{n,m}_{s\wedge\tau^{n,m}_\gamma}(i)}\le
\sum_i\sum_{k=1}^\ell
\mathbf E\bl( B^{n,m,k}_{t\wedge\tau^{n,m}_\gamma}(i)+ D^{n,m,k}_{t\wedge\tau^{n,m}_\gamma}(i)
+ M^{n,m,k}_{t\wedge\tau^{n,m}_\gamma}(i)+ \overline M^{n,m,k}_{t\wedge\tau^{n,m}_\gamma}(i)\br)
\\+
\sum_i\mathbf E
\overline F^{n,m}_{t\wedge\tau^{n,m}_\gamma}(i)+\sum_i\sum_{i'}
\mathbf E F^{n,m}_{t\wedge\tau^{n,m}_\gamma}(i',i)
+\sum_i \mathbf E E^{n,m}_{t\wedge\tau^{n,m}_\gamma}(i)\,.
\end{multline*}
Since the processes on the righthand sides of \eqref{eq:18} are local
martingales and $ X^{n,m,\ast}_{s}\le m\gamma$ when
$s<\tau^{n,m}_\gamma$\,, for some $K_5>0$\,,
\begin{equation}
  \label{eq:98}
    \mathbf E\sum_{0<
    s\le t}\abs{\Delta X^{n,m}_{s\wedge\tau^{n,m}_\gamma}(i)}\le
K_5m\gamma(1+\gamma)t\,.
\end{equation}
It follows that 
\begin{multline*}
   \mathbf E\sum_{i}X^{n,m}_{t\wedge\tau^{n,m}_\gamma}(i)^2
\mathbf 1_{\Gamma^{n,m}}
\le Rm^2+
2K_3(1+\gamma)\int_0^t\mathbf E\sum_{i}X^{n,m}_{s\wedge
  \tau^{n,m}_\gamma}(i)^2\mathbf 1_{\Gamma^{n,m}}\,ds\,
\\
+2K_3
m^2\gamma^2t+bK_5m\gamma(1+\gamma)t\,.
\end{multline*}
By Gronwall's inequality, 
\begin{equation}
  \label{eq:87}
    \mathbf E\sum_{i}X^{n,m}_{t\wedge\tau^{n,m}_\gamma}(i)^2
\mathbf 1_{\Gamma^{n,m}}
\le (Rm^2+2K_3\gamma^2tm^2+bK_5m\gamma(1+\gamma)t
)e^{2K_3(1+\gamma)t}\,.
\end{equation}
By \eqref{eq:29'} and \eqref{eq:72}, recalling \eqref{eq:46}, for some $K_6>0$\,,
\begin{multline*}
2\sum_i\abs{\int_0^tX^{n,m}_{s-}(i)
d N^{n,m}_{s}(i)}\le 
    \sum_i X^{n,m}_t(i)^2+\sum_i X^{n,m}_0(i)^2+
\\2\sum_i\abs{\int_0^tX^{n,m}_{s}(i)\,d A^{n,m}_{s}(i)}
+b\sum_{0<s\le t}\sum_i\abs{\Delta X^{n,m}_s(i)}
\\\le     \sum_i X^{n,m}_t(i)^2+\sum_i X^{n,m}_0(i)^2
+K_6\int_0^t\sum_{i}X^{n,m}_{s}(i)^2\,ds+K_6
m^2\int_0^t(R^{n,m}_{s})^2\,ds\\+
K_6\int_0^t \sum_iX^{n,m}_s(i)^2R_s^{n,m}\,ds
+b\sum_{0<s\le t}\sum_i\abs{\Delta X^{n,m}_s(i)}\,.
\end{multline*}
Therefore,  by \eqref{eq:98} and \eqref{eq:87},
\begin{multline*}
      2\mathbf E\sum_i\abs{\int_0^{t\wedge\tau^{n,m}_\gamma}X^{n,m}_{s-}(i)
d N^{n,m}_{s}(i)}\mathbf 1_{\Gamma^{n,m}}\le
(1+K_6t+K_6t\gamma )
(Rm^2+2K_3\gamma^2tm^2\\+bK_5m\gamma(1+\gamma)t)e^{2K_3(1+\gamma)t}
+Rm^2+K_6t\gamma^2m^2
+bK_5m\gamma(1+\gamma)t\,.
\end{multline*}
On applying Doob's inequality, for $L>0$ and $K>0$\,,
\begin{multline*}
  \mathbf P(\mathbf 1_{\Gamma^{n,m}}\sup_{t\le L\wedge\tau^{n,m}_\gamma }\abs{
\sum_i\int_0^tX^{n,m}_{s-}(i)\,dN^{n,m}_s(i)}>Km^2)\\
\le\frac{1}{2Km^2}\,\Bl( 
(1+K_6L+K_6L\gamma )
(Rm^2+2K_3\gamma^2Lm^2\\
+bK_5m\gamma(1+\gamma)L)e^{2K_3(1+\gamma)L}
+Rm^2+K_6L\gamma^2m^2+bK_5m\gamma(1+\gamma)L\Br)\,.
   \end{multline*}
By \eqref{eq:29'}, \eqref{eq:72}, \eqref{eq:98}, and Gronwall's inequality,
\begin{multline*}
\sup_{t\le L\wedge\tau^{n,m}_\gamma}  \sum_iX_t^{n,m}(i)^2
\le\bl( \sum_iX^{n,m}_0(i)^2+
2K_3m^2\gamma^2L\\
+2\sup_{t\le L\wedge\tau^{n,m}_\gamma}\abs{\sum_i\int_0^tX^{n,m}_{s-}(i)
d N^{n,m}_{ s}(i)}
+bK_5m\gamma(1+\gamma)L\br)e^{2K_3(1+\gamma)L}\,.
\end{multline*}
It follows that, for arbitrary $K'>0$\,,
\begin{multline*}
  \mathbf P\bl(\sup_{t\le L}\sum_{i}X^{n,m}_{t}(i)^2>K'm^2\br)
\le\mathbf P(\sum_{i}X^{n,m}_0(i)^2>Rm^2)+
 \mathbf P(\tau^{n,m}_\gamma\le L)\\+
\mathbf P\bl(\mathbf 1_{\Gamma^{n,m}}
\sup_{t\le L\wedge\tau_\gamma^{n,m}}\abs{\sum_i\int_0^tX^{n,m}_{s-}(i)
d N^{n,m}_{s}(i)}>\frac{m^2}{2}\,\bl(K' e^{-2K_3(1+\gamma)L}-R
\\-2K_3\gamma^2L-\frac{bK_5\gamma(1+\gamma)L}{m}\br)\br)
\le
\mathbf P(\sum_{i}X^{n,m}_0(i)^2>Rm^2)+
 \mathbf P(\tau^{n,m}_\gamma\le L)\\+
\frac{1}{m^2\bl(K' e^{-2K_3(1+\gamma)L}-R
-2K_3\gamma^2L-bK_5\gamma(1+\gamma)L/m\br)}\\
\Bl((1+K_6L+K_6L\gamma )
(Rm^2+2K_3\gamma^2Lm^2+bK_5m\gamma(1+\gamma)L)e^{2K_3(1+\gamma)L}\\
+Rm^2+K_6L\gamma^2m^2+bK_5m\gamma(1+\gamma)L\Br)\,.
\end{multline*}
On recalling \eqref{eq:73},
\begin{multline*}
  \lim_{K'\to\infty}\limsup_{n,m\to\infty}
  \mathbf P(\sup_{t\le L}\sum_i
X^{n,m}_{t}(i)^2>K'm^2)\\\le
\limsup_{n,m\to\infty}\mathbf P(\sum_{i}X^{n,m}_0(i)^2>Rm^2)+
\limsup_{n,m\to\infty}\mathbf P\bl(\sup_{t\le L}
 R^{n,m}_{t}\ge \gamma)\,.
\end{multline*}
Letting  $R\to\infty$\,, $\gamma\to\infty$\,, and accounting for Lemma
\ref{le:numberofgroups2} and the hypotheses of the theorem
yield the convergence
\begin{equation*}
  \lim_{K'\to\infty}\limsup_{n,m\to\infty}
\mathbf P\bl(\sup_{t\le L}\sum_{i}X^{n,m}_{t}(i)^2>K'm^2\br)=0\,.
\end{equation*}
\end{proof}

Let $a(u)$\,, where $u\in\R^\ell\,,$ 
represent a bounded and  continuously differentiable  function
 of compact support.
By  \eqref{eq:19} and   \eqref{eq:52a}, for $t\ge0$\,,
\begin{multline*}
\sum_{i}
a \bl(\frac{i}{n}\br)
\bl(  \frac{X^{n,m}_t(i)}{m}-\frac{X^{n,m}_0(i)}{m}\br)=
\int_0^t\sum_i
a \bl(\frac{i}{n}\br)\Bl(\sum_{k=1}^\ell \Bl(-
\frac{X^{n,m}_s(i)}{m}
i_k(\beta^{n,m,k}(i)+
\delta^{n,m,k}(i))\\
+
\frac{X^{n,m}_s(i-e_k)}{m}(i_k-1)\beta^{n,m,k}(i-e_k)+
\frac{X^{n,m}_s(i+e_k)}{m}(i_k+1)
\delta^{n,m,k}(i+e_k)\\
+\sum_{i'\not=i-e_k}
\frac{X^{n,m}_s(i')}{m}\,i_k'\mu^{n,m,k}(i')
\frac{X^{n,m}_s(i-e_k)}{mR^{n,m}_s}
-
\sum_{i'\not=i}
\frac{X^{n,m}_s(i')}{m}\,i_k'\mu^{n,m,k}(i')
\frac{X^{n,m}_s(i)}{mR^{n,m}_s}\\
-
\frac{X^{n,m}_s(i)}{m}\,
i_k\mu^{n,m,k}(i)
+
\frac{X^{n,m}_s(i+e_k)}{m}
(i_k+1)\mu^{n,m,k}(i+e_k)\Br)
-
\frac{X^{n,m}_s(i)}{m}\phi^{n,m}(i)\\
+
\sum_{i'}
\frac{X^{n,m}_s(i')}{m}\phi^{n,m}(i')\eta^{n,m}(i',i)
-
\frac{X^{n,m}_s(i)}{m}\,
\sum_{i'}X^{n,m}_s(i')\epsilon^{n,m}(i)\Br)\,ds
+N'^{n,m}_t\,,
\end{multline*}
where 
\begin{equation}
  \label{eq:53}
  N_t'^{n,m}=\frac{1}{m}\sum_i a \bl(\frac{i}{n}\br)
N_t^{n,m}(i)\,.
\end{equation}
Changing summation indices and regrouping
  yield
\begin{multline}
  \label{eq:64}
\sum_i a \bl(\frac{i}{n}\br)
  \frac{X^{n,m}_t(i)}{m}=
\sum_i a\bl(\frac{i}{n}\br)
  \frac{X^{n,m}_0(i)}{m}\\+
\int_0^t\sum_i\Bl(
 \frac{X^{n,m}_s(i)}{m}\sum_{k=1}^\ell \Bl(
n\bl(a \bl(\frac{i+e_k}{n}\br)-a \bl(\frac{i}{n}\br)\br)
\,\,\frac{i_k}{n}\,\beta^{n,m,k}(i)\\+
n\bl(
a \bl(\frac{i-e_k}{n}\br)-
a \bl(\frac{i}{n}\br)\br)
\frac{i_k}{n}\,\delta^{n,m,k}(i)
+n\bl(a \bl(\frac{i-e_k}{n}\br)-a \bl(\frac{i}{n}\br)
\br)\,
\frac{i_k}{n}\,\mu^{n,m,k}(i)\\+
n\bl(a \bl(\frac{i+e_k}{n}\br)-a \bl(\frac{i}{n}\br)\br)
\sum_{i'\not=i}\,\frac{i'_k}{n}\,\frac{
\mu^{n,m,k}(i')X^{n,m}_s(i')}
{mR^{n,m}_s}\Br)
\\+
a \bl(\frac{i}{n}\br)\sum_{i'}
\frac{X^{n,m}_s(i')}{m}\,\phi^{n,m}(i')\eta^{n,m}(i',i)
-a \bl(\frac{i}{n}\br)\phi^{n,m}(i)\,\frac{X^{n,m}_s(i)}{m}\\-
a \bl(\frac{i}{n}\br)\,\frac{X^{n,m}_s(i)}{m}
\sum_{i'}X^{n,m}_s(i')\epsilon^{n,m}(i)\Br)\,ds+N'^{n,m}_t\,.
\end{multline}
Owing to \eqref{eq:57a} and \eqref{eq:53},
 \begin{multline*}
N_t'^{n,m}
\\=        
\frac{1}{m}\sum_i\Bl(\sum_{k=1}^\ell \Bl(
\bl(a \bl(\frac{i+e_k}{n}\br)-
a \bl(\frac{i}{n}\br)\br)N_t^{B,n,m,k}(i)
+\bl(a \bl(\frac{i-e_k}{n}\br)-a \bl(\frac{i}{n}\br)\br)
 N^{D,n,m,k}_t(i)\\
+\bl(a \bl(\frac{i+e_k}{n}\br)-
a \bl(\frac{i}{n}\br)\br)
 N^{M,n,m,k}_t(i)
+\br(a \bl(\frac{i-e_k}{n}\br)-
a \bl(\frac{i}{n}\br)\br)
 N^{\overline M,n,m,k}_t(i)\Br)
\\+
\sum_{i'}a \bl(\frac{i}{n}\br)
 N^{F,n,m}_t(i',i)-
a \bl(\frac{i}{n}\br)
\bl( N^{\overline{F},n,m}_t(i)+ N^{E,n,m}_t(i)\br)\Br)
\,.
\end{multline*}
Thus,  $ N'^{n,m}=(N_t'^{n,m}\,,t\ge0)$ is 
 a locally square integrable martingale  with the predictable quadratic
variation process
\begin{multline*}
  \langle N'^{n,m}\rangle_t=
\frac{1}{m^2}\sum_i\Bl(\sum_{k=1}^\ell \Bl(
\bl(a \bl(\frac{i+e_k}{n}\br)-
a \bl(\frac{i}{n}\br)\br)^2\langle N^{B,n,m,k}(i)\rangle_t\\
+\bl(a \bl(\frac{i-e_k}{n}\br)-a \bl(\frac{i}{n}\br)\br)^2
 \langle N^{D,n,m,k}(i)\rangle_t\\
+\bl(a \bl(\frac{i+e_k}{n}\br)-
a \bl(\frac{i}{n}\br)\br)^2
\langle N^{M,n,m,k}(i)\rangle_t
\\
+\br(a \bl(\frac{i-e_k}{n}\br)-
a \bl(\frac{i}{n}\br)\br)^2
 \langle N^{\overline M,n,m,k}(i)\rangle_t
\\
+\sum_{i'}
\br(a \bl(\frac{i+e_k}{n}\br)-
a \bl(\frac{i}{n}\br)\br)\br(a \bl(\frac{i'-e_k}{n}\br)-
a \bl(\frac{i'}{n}\br)\br)
 \langle N^{ M,n,m,k}(i),N^{\overline M,n,m,k}(i')\rangle_t\Br)\\
+
a \bl(\frac{i}{n}\br)^2\bl(\sum_{i'}
 \langle N^{F,n,m}(i',i)\rangle_t+
 \langle N^{E,n,m}(i)\rangle_t+\langle N^{\overline{F},n,m}(i)\rangle_t\br)
\\
+a \bl(\frac{i}{n}\br)\sum_{i'}\sum_ja \bl(\frac{j}{n}\br)
\langle N^{F,n,m}(i',i),N^{F,n,m}(i',j)\rangle_t
\\-a \bl(\frac{i}{n}\br)\sum_{i'}a \bl(\frac{i'}{n}\br)
 \langle N^{F,n,m}(i,i'),N^{\overline{F},n,m}(i)\rangle_t
\Br)\,.
\end{multline*}
Substitutions from \eqref{eq:28a} and  \eqref{eq:66} with the account
of \eqref{eq:90}    yield
\begin{multline}
  \label{eq:61}
  \langle N'^{n,m}\rangle_t=  \frac{1}{m^2}
\sum_i\Bl(\sum_{k=1}^\ell \Bl(
\bl(a \bl(\frac{i+e_k}{n}\br)-
a \bl(\frac{i}{n}\br)\br)^2
\int_0^tX^{n,m}_s(i)i_k\beta^{n,m,k}(i)\,ds\\
+
\bl(a \bl(\frac{i-e_k}{n}\br)-a \bl(\frac{i}{n}\br)\br)^2
 \int_0^tX^{n,m}_s(i)i_k\delta^{n,m,k}(i)\,ds\\+
\bl(a \bl(\frac{i+e_k}{n}\br)-
a \bl(\frac{i}{n}\br)\br)^2
\int_0^t\sum_{i'\not=i
}
X^{n,m}_s(i')i'_k\mu^{n,m,k}(i')
\frac{X^{n,m}_s(i)}{X^{n,m,\ast}_s}\,ds
\\
+\br(a \bl(\frac{i-e_k}{n}\br)-
a \bl(\frac{i}{n}\br)\br)^2
\int_0^tX^{n,m}_s(i)i_k\mu^{n,m,k}(i)\bl(1-\frac{X^{n,m}_s(i)}{X^{n,m,\ast}_s}\br)\,ds
\\
+\sum_{i'\not=i}
\br(a \bl(\frac{i+e_k}{n}\br)-
a \bl(\frac{i}{n}\br)\br)\br(a \bl(\frac{i'-e_k}{n}\br)-
a \bl(\frac{i'}{n}\br)\br)
 \int_0^tX^{n,m}_s(i')
i'_k\mu^{n,m,k}(i')
\frac{X_s^{n,m}(i)}{X_s^{n,m,\ast}}\,ds\Br)\\
+\sum_{i'}a \bl(\frac{i}{n}\br)^2
 \int_0^t
X^{n,m}_s(i')\phi^{n,m}(i')\alpha^{n,m}(i',i)\,ds
\\+a \bl(\frac{i}{n}\br)^2
\int_0^t
X^{n,m}_s(i)\bl(X^{n,m,\ast}_s\epsilon^{n,m}(i)+\phi^{n,m}(i)\br)\,ds
\\+
\int_0^tX_s^{n,m}(i)\phi^{n,m}(i)
\mathbf E\bl(\sum_{i'}a \bl(\frac{i'}{n}\br)\theta^{n,m}_i(i',1)\br)^2
\,ds
\\-a \bl(\frac{i}{n}\br)\sum_{i'}a \bl(\frac{i'}{n}\br)
\int_0^tX_s^{n,m}(i)\eta^{n,m}(i,i')\phi^{n,m}(i)\,ds
\Br)\,.
\end{multline}

Let
\begin{equation}
  \label{eq:79}
    a^n(u)=a \bl(\frac{\lfloor nu\rfloor}{n}\br)
\end{equation}
and
\begin{equation}
  \label{eq:65}
  Y^{n,m}_t=
\int_{\R_+^\ell}
a^n(u)
\Lambda^{n,m}_t(du)\,.
\end{equation}
By  \eqref{eq:2}, \eqref{eq:5}, \eqref{eq:32},   \eqref{eq:64}, and \eqref{eq:61}, 
\begin{equation}
  \label{eq:9}
  Y_t^{n,m}=Y^{n,m}_0+\int_0^t Z^{n,m}_s\,ds
+N'^{n,m}_t\,,
\end{equation}
where
\begin{multline}
  \label{eq:11}
  Z^{n,m}_s=
\int_{\R_+^\ell}\Bl(
\sum_{k=1}^\ell \Bl(
n\bl(a^n\bl(u+\frac{e_k}{n}\br)-a ^n(u)\br)
\,
\frac{\lfloor nu_k\rfloor}{n}\,\hat\beta^{n,m,k}(u)\\+
n\bl(
a^n\bl(u-\frac{e_k}{n}\br)-
a^n(u)\br)\,
\frac{\lfloor nu_k\rfloor}{n}\,(\hat\delta^{n,m,k}(u)+
\hat\mu^{n,m,k}(u))
\\+
n\bl(a^n\bl(u+\frac{e_k}{n}\br)-a^n(u)\br)
\displaystyle\int_{\R_+^\ell}
\frac{1}{ R^{n,m}_s}\,
\,\frac{\lfloor nu_k'\rfloor}{n}\,
\hat\mu^{n,m,k}(u')
\,\Lambda^{n,m}_s(du')\Br)
\\+
\int_{\R_+^\ell}
a^n(u')\hat\phi^{n,m}(u)\hat\eta^{n,m}(u,du')
-a^n(u)
\bl(\hat\phi^{n,m}(u)+
R^{n,m}_s\hat\epsilon^{n,m}(u)\br)\Br)\Lambda^{n,m}_s(du)\\
-\int_0^t\sum_i\sum_{k=1}^\ell
 \bl(a \bl(\frac{i+e_k}{n}\br)-a (\frac{i}{n})\br)
\frac{1}{ R^{n,m}_s}\,
i_k\mu^{n,m,k}(i)\frac{X^{n,m}_s(i)^2}{m^2}\,ds
\end{multline}
and,  by \eqref{eq:61}, on letting $B_{1/n}(u)=\{u'\in\R_+^\ell:\,
\max_{k=1,\ldots,\ell}\abs{u'_k-u_k}< 1/n\}$\,,
\begin{multline*}
        \langle N'^{n,m}\rangle_t=  \frac{1}{m}\,
\int_0^t\int_{\R_+^\ell} \Bl(\sum_{k=1}^\ell \Bl(
n\bl(a^n\bl(u+\frac{e_k}{n}\br)-
a^n(u)\br)^2
\frac{\lfloor nu_k\rfloor}{n}\,\hat\beta^{n,m,k}(u)\\
+n
\bl(a^n\bl(u-\frac{e_k}{n}\br)
-a^n(u)\br)^2
\frac{\lfloor nu_k\rfloor}{n}\,\hat\delta^{n,m,k}(u)\\
+n\bl(a^n\bl(u+\frac{e_k}{n}\br)
-a^n(u)\br)^2
\int_{\R_+^\ell\setminus B_{1/n}(u)}
\frac{1}{R^{n,m}_s}\,
\frac{\lfloor nu'_k\rfloor}{n}\,\hat\mu^{n,m,k}(u')
\,\Lambda^{n,m}_s(du')\\
+n\bl(a^n\bl(u-\frac{e_k}{n}\br)
-a^n(u)\br)^2
\frac{\lfloor nu_k\rfloor}{n}\,
\hat\mu^{n,m,k}(u)\bl(1-\frac{\hat X^{n,m}_s(u)}{R^{n,m}_s}\br)\\
+n\bl(a^n\bl(u+\frac{e_k}{n}\br)
-a^n(u)\br)\frac{1}{R^{n,m}_s}
\\
\int_{\R_+^\ell\setminus B_{1/n}(u)}
\bl(a^n\bl(u'-\frac{e_k}{n}\br)
-a^n(u')\br)
\frac{\lfloor
  nu_k'\rfloor}{n}\,\hat\mu^{n,m,k}(u')
\,\Lambda^{n,m}_s(du')\Br)
\\+\hat\phi^{n,m}(u)
\int_{\R_+^\ell}a^n(u')^2
\hat\alpha^{n,m}(u,du')
+a^n(u)^2\bl(
R^{n,m}_s\hat\epsilon^{n,m}(u)+\hat\phi^{n,m}(u)\br)\\
+\hat\phi^{n,m}(u)
\mathbf E\bl(\sum_{i'}a \bl(\frac{i'}{n}\br)
\theta^{n,m}_{\lfloor nu\rfloor}(i',1)\br)^2
\\-a^n(u)
\hat\phi^{n,m}(u)\int_{\R_+^\ell}
a^n(u')\hat\eta^{n,m}(u,du')
\Br)
\Lambda^{n,m}_s(du)\,ds
\,,
\end{multline*}
where, in analogy with \eqref{eq:32},
$\hat\alpha^{n,m}(u,\Gamma)=\sum_{i'}\alpha^{n,m}(\lfloor nu\rfloor,i'/n)
\mathbf 1_{\Gamma}(i'/n)$\,.
We note that
$\mathbf E\bl(\sum_{i'}a (i'/n)
\theta_{\lfloor nu\rfloor}^{n,m}(i',1)\br)^2
\le b^2\sup_{u}a (u)^2$\, and, by \eqref{eq:104},
\begin{equation*}
    \hat\alpha^{n,m}(u,\R_+^\ell)=
\sum_{i'}\alpha^{n,m}\bl(\lfloor nu\rfloor,i'\br)\le
b\eta^{n,m}(\lfloor nu\rfloor,\R_+^\ell)\le b^2\,.
\end{equation*}
Therefore, on recalling \eqref{eq:19},
 \eqref{eq:79}, the fact that the
$a (u)$ are differentiable of bounded support and the boundedness
hypotheses of the theorem, we have that, given $L>0$\,,
 for some $\hat K>0$\,, which may depend
on $L$\,, for $t\le L$\,,
\begin{equation*}
  \langle N'^{n,m}\rangle_t\le \frac{\hat K}{m}\,
\int_0^t(1+(R^{n,m}_s)^2)\,ds\,.
\end{equation*}
By Lemma \ref{le:numberofgroups}, Lemma \ref{le:yule} and the Lenglart--Rebolledo
inequality, for $\chi>0$\,,
\begin{equation}
  \label{eq:81}
  \lim_{n,m\to\infty}\mathbf P(\sup_{t\le L}\abs{N'^{n,m}_t}>\chi)=0\,.
\end{equation}

\begin{lemma}
  \label{le:L^1}
The sequence $\Lambda^{n,m}$ is $C$-tight for convergence in
distribution  in\\ $\mathbb D(\R_+,\mathbb M_+(\R_+^\ell))$\,.
\end{lemma}
\begin{proof}
  By Theorem 4.6 in Jakubowski \cite{Jak86} and Tops\oe\,\cite{Top70}
  (or Tops\oe\,\cite{Top}),
it is sufficient to prove
that, for all $L>0$ and $\gamma>0$\,,
\begin{equation*}
  \lim_{K\to\infty}\limsup_{n,m\to\infty}
\mathbf
P\bl(\sup_{t\in[0,L]}\Lambda^{n,m}_t(\R_+^\ell)
>K\br)=0\,,
\end{equation*}
\begin{equation*}
  \lim_{K\to\infty}\limsup_{n,m\to\infty}
\mathbf P\bl(\sup_{t\in[0,L]}\Lambda_t^{n,m}(u:\,\abs{u}>K)>\gamma\br)=0
\end{equation*}
and, for all continuous functions $g$ of compact support  and all $\gamma>0$\,,
\begin{equation*}
 \lim_{\chi\to0} \limsup_{n,m\to\infty}
\mathbf
P(\sup_{\substack{s,t\in[0,L]:\\\abs{s-t}\le\chi}}
\abs{\int\limits_{\R_+^\ell}g(u)(
\Lambda^{n,m}_s(du)-\Lambda^{n,m}_{t}(du))}>\gamma)=0\,.
\end{equation*}
The first and second requirements are fulfilled by Lemma \ref{le:yule}
(see \eqref{eq:30} and \eqref{eq:20}). 
Let us note that, by \eqref{eq:11}, \eqref{eq:19},
Lemma \ref{le:numberofgroups} and Lemma
\ref{le:yule},
\begin{equation*}
  \lim_{K\to\infty}\limsup_{n,m\to\infty}
\mathbf P(\sup_{s\le t}\abs{Z^{n,m}_s}>K)=0\,.
\end{equation*}
Therefore, the third
limit follows from \eqref{eq:65}, \eqref{eq:9}, \eqref{eq:11},  and
\eqref{eq:81}.
\end{proof}
We now identify limit points of the $\Lambda^{n,m}$\,. 
Let $(\tilde \lambda_t\,,t\ge0)$
represent a limit point in distribution of 
$\Lambda^{n,m}$ along a
subsequence in 
 $\mathbb D(\R_+,\mathbb M_+(\R_+^\ell))$\,. 
 We keep the notation $(n,m)$ for the
subsequence. By Lemma \ref{le:L^1}, $\tilde\lambda_t$ is continuous in
$t$ for the metric of weak convergence in $\mathbb M_+(\R_+^\ell)$\,.
Since the functions $a ^n(u)$ are bounded uniformly in $u$ and $n$ and
converge to $a (u)$  uniformly in $u$\,, see \eqref{eq:79},
 by \eqref{eq:65} and the continuous mapping theorem, in distribution in $\mathbb D(\R_+,\R)$\,,
\begin{equation}
  \label{eq:8}
  Y^{n,m}_t\to Y_t=\int_{\R_+^\ell} a (u)\,\tilde \lambda_t(du)\,.
\end{equation}
On recalling \eqref{eq:19} and Lemma \ref{le:numberofgroups}, we obtain that,
in distribution in $\mathbb D(\R_+,\R)$\,,
\begin{equation}
  \label{eq:60}
    R^{n,m}_s\to R_s=
\tilde \lambda_s(\R_+^\ell)\,.
\end{equation}
By Lemma \ref{le:numberofgroups}, the latter quantity is
bounded away from zero locally uniformly in $s$
with probability 1.
Since the function $a (u)$ is continuously differentiable and is of
compact support in $u$\,,
$a^n(u)\to a (u)$  and
$n\bl(a ^n\bl(u\pm e_k/n\br)-a ^n(u)\br)\to\pm \partial_{u_k}
a (u)$ uniformly in $u$\,, as
$n\to\infty$\,. 
Therefore, 
 in distribution in $\mathbb D(\R_+,\R)$\,, for $k=1,\ldots,\ell$\,,
  \begin{equation}
  \label{eq:50}
  \int_{\R_+^\ell}  n\bl(a^n\bl(  u+\frac{e_k}{n}\br)-a^n(u)\br)
\Lambda^{n,m}_s(du)\to
\int_{\R_+^\ell}\partial_{u_k}a (u)
\tilde \lambda_s(du)\,,
\end{equation}
and, since the $u_k\hat\mu^{n,m,k}(u)$ are bounded and the
convergences in \eqref{eq:7} hold, 
\begin{equation}
  \label{eq:43}
    \int_{\R_+^\ell}
\frac{\lfloor nu_k\rfloor}{n}\,
\hat\mu^{n,m,k}(u)
\Lambda^{n,m}_s(du)\to 
\int_{\R_+^\ell}
\,u_k \hat\mu^k(u)
\,\tilde \lambda_s(du)\,.
\end{equation}
Since the convergences in \eqref{eq:60},
 \eqref{eq:50}, and \eqref{eq:43}
 hold jointly,   in distribution 
in $\mathbb D(\R_+,\R)$\,,
\begin{multline}
  \label{eq:59}
    \int_{\R_+^\ell}  n\bl(a^n\bl(  u+\frac{e_k}{n}\br)-a^n(u)\br)\,\Lambda^{n,m}_s(du)\int_{\R_+^\ell}
\frac{1}{ R^{n,m}_s}\,
\,\frac{\lfloor nu_k'\rfloor}{n}\,
\hat\mu^{n,m,k}(u')
\,\Lambda^{n,m}_s(du')\\
\to
\frac{\displaystyle\int_{\R_+^\ell}
u_k\hat\mu^{k}(u)
\,\tilde \lambda_s(du)}{\displaystyle\tilde \lambda_s(\R_+^\ell)\,}
  \int_{\R_+^\ell}\partial_{u_k}a(u)
\tilde \lambda_s(du)\,.
\end{multline}
Similar lines of reasoning show that, jointly in distribution
in $\mathbb D(\R_+,
\R^4)$\,, and jointly with the convergence in \eqref{eq:59}, for $k=1,\ldots,\ell$\,,
\begin{align*}
\int_{\R_+^\ell}
n\bl(a^n\bl(  u+\frac{e_k}{n}\br)-a^n(u)\br)
\,\frac{\lfloor
  nu_k\rfloor}{n}\,\hat\beta^{n,m,k}(u)
\,\Lambda^{n,m}_s(du)\\\to
\int_{\R_+^\ell}\partial_{u_k}a (u)\,u_k\hat\beta^{k}(u)\,
\tilde \lambda_s(du),\\
\int_{\R_+^\ell}
n\bl(
a^n(u)-a^n\bl(u-\frac{e_k}{n}\br)
\br)\,\frac{\lfloor nu_k\rfloor}{n}\,\hat\delta^{n,m,k}(u)\,
\Lambda^{n,m}_s(du)
\\\to
\int_{\R_+^\ell}
\partial_{u_k}a (u)
u_k\hat\delta^{k}(u)\tilde\lambda_s(du)\,,
\\\int_{\R_+^\ell}
n\bl(a^n\bl(  u+\frac{e_k}{n}\br)-a^n(u)\br)
\Lambda^{n,m}_s(du)
\displaystyle\int_{\R_+^\ell}
\frac{1}{ R^{n,m}_s}\,
\,\frac{\lfloor nu_k'\rfloor}{n}\,
\hat\mu^{n,m,k}(u')
\,\Lambda^{n,m}_s(du')\\
\to \frac{\displaystyle
\int_{\R_+^\ell}
u_k\hat\mu^{k}(u)
\,\tilde \lambda_s(du)
}{\displaystyle\tilde \lambda_s(\R_+^\ell)}
\int_{\R_+^\ell}
\partial_{u_k}a (u)\tilde \lambda_s(du)\,,
\\\int_{\R_+^\ell}
n\bl(a^n(u)
-a^n\bl(u-\frac{e_k}{n}\br)\br)\,
\frac{\lfloor nu_k\rfloor}{n}\,\hat\mu^{n,m,k}(u)\,
\Lambda^{n,m}_s(du)\\
\to
\int_{\R_+^\ell}\partial_{u_k}a (u)
u_k\hat\mu^{k}(u)\,\tilde \lambda_s
(du)\,,
\end{align*}
and
\begin{align*}\int_{\R_+^\ell}\bl(\hat\phi^{n,m}(u)
\int_{\R_+^\ell}a^n(u')
\hat\eta^{n,m}(u,du')-a^n(u)
\bl(\hat\phi^{n,m}(u)+
R^{n,m}_s\hat\epsilon^{n,m}(u)\br)\br)\,
\,\Lambda^{n,m}_s(du)\\
\to \int_{\R_+^\ell}\bl(
\int_{\R_+^\ell}a (u')
\hat\varphi(u,du')
-a (u)
\bl(\hat\phi(u)+
R_s\hat\epsilon(u)\br)\br)\,
\tilde \lambda_s(du)\,.
\end{align*}
Let us note also that, by Lemma \ref{le:dyra}, for arbitrary $\gamma>0$\,,
\begin{equation*}
\lim_{n,m\to\infty}
\mathbf P\bl(\sup_{s\le t}\sum_i\sum_{k=1}^\ell\abs{
  a \bl(\frac{i+e_k}{n}\br)-a \bl(\frac{i}{n}\br)}
i_k\,\mu^{n,m,k}(i)\hat X^{n,m}_s(i)^2>\gamma)=0\,.
\end{equation*}
Therefore, by  \eqref{eq:11}, in distribution in $\mathbb
D(\R_+,\R)$\,,
\begin{equation}
  \label{eq:68}
  Z_s^{n,m}\to Z_s\,,
\end{equation}
where
\begin{multline*}
   Z_s=\sum_{k=1}^\ell \Bl(\int_{\R_+^\ell}
\partial_{u_k}a (u)
\,u_k(\hat\beta^{k}(u)-\hat\delta^{k}(u)-\hat\mu^k(u))\,
\tilde \lambda_s(du)\\
+\frac{\displaystyle\int_{\R_+^\ell} 
u_k\,\hat\mu^{k}(u)\,\tilde \lambda_s(du)}{
\displaystyle\tilde \lambda_s(\R_+^\ell)}
  \int_{\R_+^\ell} \partial_{u_k} a (u)
\,\tilde \lambda_s(du)\Br)
\\+\int_{\R_+^\ell}\bl(
\int_{\R_+^\ell}
a (u')\hat\varphi(u,du')
-a (u)\hat\phi(u)
-\tilde \lambda_s(\R_+^\ell)
a (u)\,\hat\epsilon(u)\br)\tilde \lambda_s(du)\,.
\end{multline*}
Since $Z_s$ has continuous trajectories owing to Lemmas
\ref{le:numberofgroups} and \ref{le:yule},
by \eqref{eq:9},  \eqref{eq:81}, \eqref{eq:8}, and \eqref{eq:68},
\begin{equation*}
  Y_t=Y_0+\int_0^tZ_s\,ds\,.
\end{equation*}
By continuity of $Z_s$\,, 
 $Y_t$ is  continuously differentiable with respect to $t$ and
 \begin{multline}
   \label{eq:4}
\frac{d}{dt}\,\int_{\R_+^\ell} a (u)\,\tilde
\lambda_t(du)=
\sum_{k=1}^\ell \Bl(\int_{\R_+^\ell}
\partial_{u_k}a (u)
u_k(\hat\beta^{k}(u)-\hat\delta^{k}(u)-\hat\mu^k(u))
\tilde\lambda_t(du)\\
+\frac{\displaystyle\int_{\R_+^\ell} u_k\,\hat\mu^{k}(u)\,\tilde\lambda_t(du)}{
\displaystyle\tilde \lambda_t(\R_+^\ell)}
  \int_{\R_+^\ell} \partial_{u_k} a (u)\tilde \lambda_t(du)\Br)
\\+\int_{\R_+^\ell}\bl(
\int_{\R_+^\ell}
a (u')\hat\varphi(u,du')
-a (u)\hat\phi(u)
-\tilde \lambda_t(\R_+^\ell)
a (u)\,\hat\epsilon(u)\br)\tilde\lambda_t(du)\,.
 \end{multline}
We now prove that $\tilde \lambda_t$ is specified uniquely.
Given $\nu\in\mathbb M_+(\R_+^\ell)$
such that $\nu(\R_+^\ell)>0$\,, we define, for 
$y\in\mathbb C^1_c(\R_+^\ell)$\,,
  \begin{multline*}
      A(\nu)y(u)=\sum_{k=1}^\ell \Bl(   
u_k(\hat\beta^{k}(u)-\hat\delta^{k}(u)-\hat\mu^k(u))
+\frac{\displaystyle\int_{\R_+^\ell} u_k'\,\hat\mu^{k}(u')
\nu(du')}{
\displaystyle\nu(\R_+^\ell)}\,\Br)\partial_{u_k} y(u)\,.
\end{multline*}
We also let, for bounded functions $y$\,,
\begin{equation}
  \label{eq:34}
      B(\nu) y(u)=
\int_{\R_+^\ell}
y(u')\hat\varphi(u,du')
-\bl(\hat\phi(u)+
\hat\epsilon(u)\nu(\R_+^\ell)\br)y(u)\,.
\end{equation}
We  write 
 \eqref{eq:4}  as 
 \begin{equation}
\label{eq:3}      \frac{d}{ds}\,\langle a,\tilde\lambda_s\rangle =
\langle   (A( \tilde \lambda_s) +B(\tilde \lambda_s))a,\tilde
\lambda_s\rangle\,.
 \end{equation}
(In the rest of the section,  $\langle\cdot,\cdot\rangle$ represents
 the pairing between $\mathbb L^\infty(\R_+^\ell)$ and
$\mathbb M(\R_+^\ell)$\,, with $\mathbb M(\R_+^\ell)$ denoting the set
of signed Borel measures on $\R_+^\ell$ with the total variation norm.)
Let $a_s(u)$\,, where $ s\in\R,\,u\in\R^\ell\,,$ 
represent a bounded and  continuously differentiable  function
 compactly supported in $u$  uniformly over
 $s$ from bounded intervals. Noting that
 \begin{equation*}
\frac{1}{\Delta s}\bl(\langle a_{s+\Delta s},\tilde\lambda_{s+\Delta s}\rangle-
\langle a_s,\tilde\lambda_s\rangle\br)
=\frac{1}{\Delta s}\langle a_{s+\Delta s}- a_s,
\tilde\lambda_{s+\Delta s}\rangle+
\frac{1}{\Delta s}\langle a_s,\tilde\lambda_{s+\Delta
s}- \tilde\lambda_s\rangle\,,
    \end{equation*}
letting $\Delta s\to0$
and recalling that $\tilde\lambda_s$ is continuous yield, by \eqref{eq:3}, cf. Luo and
Mattingly \cite{LuoMat17},
\begin{equation}
  \label{eq:45}
   \frac{d}{ds}\,\langle a_s,\tilde\lambda_s\rangle =
\langle \partial_sa_s+   (A( \tilde \lambda_s) +B(\tilde \lambda_s))a_s,\tilde
\lambda_s\rangle\,.
\end{equation}
Let
\begin{equation}
  \label{eq:40}
  \tilde F_t^k(u)=u_k(\hat\beta^{k}(u)-
\hat\delta^{k}(u)-\hat\mu^k(u))
+\frac{1}{
\displaystyle\tilde\lambda_t(\R_+^\ell)}\,
\displaystyle\int_{\R_+^\ell} 
u'_k\,\hat\mu^{k}(u')\,\tilde\lambda_t(du')\,,
\end{equation}
$\tilde F_t(u)=(\tilde F_t^1(u),\ldots,\tilde F_t^\ell(u))$
and let
$\tilde\psi_{s,t}(\nu,u)=\bl(\tilde\psi_{s,t}^{1}(u),\ldots,\tilde
\psi_{s,t}^{\ell}(u)\br)$\,,
where    $s\in\R$ and $t\in\R$\,, be defined by
 $\tilde\psi_{s,s}(u)=u$  and by
\begin{equation}
  \label{eq:44}
    \partial_t\tilde\psi_{s,t}(u)=
\tilde F_t(\tilde \psi_{s,t}(u))\,.
\end{equation}
Since
$u_k\hat\beta^k(u)$\,, $u_k\hat\delta^k(u)$\,, and $u_k\hat\mu^k(u)$ are
$\mathbb C^1$--functions with bounded derivatives 
and the ratio on the  righthand side of \eqref{eq:40} is
a continuous function of $t$\,, the function $\tilde\psi_{s,t}(\nu,u)$ is
continuously differentiable in $(s,t,u)$ with 
the   $u$--derivatives being  uniformly bounded
locally uniformly in $(s,t)$, see, e.g., Theorem 3.1 on
p.95 in Hartman \cite{Har64}.
Let,  for $f\in\mathbb \R^{\mathbb \R_+^\ell}$\,,
\begin{equation}
  \label{eq:67}
\tilde U_{s,t} f(u)=f(\tilde\psi_{s,t}(u))\,,   
\end{equation}
where 
 $u\in\R_+^\ell$\,.
For $f\in\mathbb C^1_c(\R_+^\ell)$\,, $\tilde U_{s,t}f(u)$ is continuously
differentiable in $(s,t,u)$ and
\begin{equation}
      \label{eq:37}
      \partial_s\tilde U_{s,t}f=-A(\tilde\lambda_s)\tilde U_{s,t}f\,,
\end{equation}
the temporal derivatives on the lefthand side being for the $\sup$--norm.
 (One way to ascertain the equation is to use  the flow
property that $\psi_{s,t}=\psi_{r,t}\circ \psi_{s,r}$\,.)
We now draw on Luo and Mattingly \cite{LuoMat17} by letting in
\eqref{eq:45}, for $t$ fixed,
$a_s=\tilde U_{s,t}f$\,, where $f\in
\mathbb C_c^1(\R_+^\ell)$\,. 
By \eqref{eq:67} and 
\eqref{eq:37}, $\partial_s a_s=-A(\tilde\lambda_s)a_s$\,, $a_t=f$\,,
and $a_0=\tilde U_{0,t}f$\,.
By \eqref{eq:45} and the fact that $\tilde\lambda_0=\hat\lambda_0$\,,
for all $0\le s\le t$\,,
\begin{equation}
  \label{eq:29}
  \langle f,\tilde\lambda_t\rangle=\langle
  \tilde U_{0,t}f,\hat\lambda_0\rangle
+\int_0^t
\langle B(\tilde \lambda_s)\tilde U_{s,t}f,\tilde
\lambda_s\rangle\,ds\,.
\end{equation}
Via limits akin to the
construction of the Daniell integral, this equality extends to bounded Borel
functions $f$\,.
Let $\breve\lambda_t$ represent another limit point of
$\Lambda^{n,m}$ and let $\breve \psi_{s,t}$ and 
$\breve U_{s,t}$ be defined in analogy with
$\tilde\psi_{s,t}$ and $\tilde U_{s,t}$\,, respectively, with $\breve\lambda_t$ as $\tilde\lambda_t$\,, so,
\begin{equation*}
    \langle f,\breve\lambda_t\rangle=\langle
  \breve U_{0,t}f,\hat\lambda_0\rangle
+\int_0^t
\langle B(\breve \lambda_s)\breve U_{s,t}f,\breve
\lambda_s\rangle\,ds\,.
\end{equation*}
We have that
\begin{multline}
  \label{eq:36}
    \langle f,\tilde\lambda_t-\breve\lambda_t\rangle=
\langle
\tilde U_{0,t}f-  \breve U_{0,t}f,\hat\lambda_0\rangle+
\int_0^t \langle B(\tilde\lambda_s)\tilde U_{s,t}f,
\tilde\lambda_s-\breve\lambda_s\rangle ds\\
+\int_0^t \langle( B(\tilde\lambda_s)- B(\breve\lambda_s))
\tilde U_{s,t}f,\breve\lambda_s\rangle ds
+\int_0^t \langle B(\breve\lambda_s)
(\tilde U_{s,t}f-\breve U_{s,t}f)\,,\breve\lambda_s\rangle ds\,.
\end{multline}
Let $\rho_w$ represent the Lipschitz metric for the weak topology
on $\R_+^\ell$\,, see, e.g., Dudley \cite{Dud89}. Since, by Lemma
\ref{le:numberofgroups} and \eqref{eq:19}, we may assume that
$\tilde\lambda_t(\R_+^\ell)$ and $\breve\lambda_t(\R_+^\ell)$ are
locally bounded away from zero,
 since the $u'_k\hat\mu_k(u')$ are bounded and Lipschitz--continuous, and
since the derivatives of the $u_k(\hat\beta^{k}(u)-
\hat\delta^{k}(u)-\hat\mu^k(u))$ are bounded, by
\eqref{eq:40}, given $T>0$\,, there exists  $\kappa>0$ such that, for
all $t\in[0,T]$ and all $u\in\R_+^\ell$\,,
\begin{equation*}
  \norm{\tilde F^k_t(u)-\breve F^k_t(u')}\le
 \kappa(\abs{u-u'}+\rho_w(\tilde\lambda_t,\breve
  \lambda_t))\,. 
\end{equation*}
By \eqref{eq:44} and the analogue for $\breve\psi_{s,t}$\,,
\begin{multline*}
      \tilde\psi_{s,t}(u)-\breve\psi_{s,t}(u)=
\int_s^t(\tilde F_{s'}(\tilde \psi_{s,s'}(u))-
\breve F_{s'}(\breve \psi_{s,s'}(u)))\,ds'\\=
\int_s^t(\tilde F_{s'}(\tilde \psi_{s,s'}(u))-
\tilde F_{s'}(\breve \psi_{s,s'}(u)))\,ds'+
\int_s^t (\tilde F_{s'}(\breve \psi_{s,s'}(u))-
\breve F_{s'}(\breve \psi_{s,s'}(u)))\,ds'\,.
\end{multline*}
Therefore,
\begin{equation}
  \label{eq:15}
    \abs{\tilde\psi_{s,t}(u)-\breve\psi_{s,t}(u)}\le
\kappa\int_s^t\abs{\tilde \psi_{s,s'}(u)-
\breve \psi_{s,s'}(u)}\,ds'+
\kappa\int_s^t\rho_w(\tilde\lambda_{s'},\breve\lambda_{s'})\,ds'\,,
\end{equation}
so, by Gronwall's inequality,
\begin{equation}
  \label{eq:17}
      \abs{\tilde\psi_{s,t}(u)-\breve\psi_{s,t}(u)}\le
\kappa 
e^{\kappa(t-s)}\int_s^t\rho_w(\tilde\lambda_{s'},\breve\lambda_{s'})\,ds'\,.
\end{equation}
Suppose that  $f$ is bounded above by one in absolute value
and is Lipschitz continuous with a Lipschitz constant
 one.
By \eqref{eq:67} and \eqref{eq:15},  for the $\sup$--norm on $\R^{\R_+^\ell}$\,,
\begin{equation}
  \label{eq:33}
  \norm{\tilde U_{s,t}f-\breve U_{s,t}f}\le 
\kappa
e^{\kappa(t-s)}\int_s^t\rho_w(\tilde\lambda_{s'},\breve\lambda_{s'})\,ds'\,.
\end{equation}
By  \eqref{eq:34}, by \eqref{eq:36}, by \eqref{eq:33},  by
 $U_{s,t}$ being  a contraction for the $\sup$--norm, and by the
 definition of $\rho_w$\,,
 on recalling that $\sup_{u\in\R_+^\ell}\hat\varphi(u,\R_+^\ell)<\infty$\,,
we have that  there exists $K>0$ such that, for all $t\le T$\,,   
\begin{equation*}
\langle f,\tilde\lambda_t-\breve\lambda_t\rangle\le  
K\int_0^t\rho_w(\tilde
\lambda_s,\breve\lambda_s)\,ds\,,
\end{equation*}
so, on maximising over $f$\,,
\begin{equation*}
  \rho_w(\tilde\lambda_t,\breve\lambda_t)
\le
K\int_0^t\rho_w(\tilde
\lambda_s,\breve\lambda_s)\,ds\,.
\end{equation*}
By Gronwall's inequality, $\tilde\lambda_t=\breve\lambda_t$\,. Thus,
\eqref{eq:4} has a unique solution, which concludes the
convergence proof. We have also proved that $\hat\lambda_t$ is
specified uniquely by \eqref{eq:4} which is the same equation as \eqref{eq:13}.
Besides, we  can and will refer to $\tilde F_t$\,, $\tilde \psi_{s,t}$\,,
and
 $\tilde U_{s,t}$ as 
$ F_t$\,, $ \psi_{s,t}$\,,
and
 $ U_{s,t}$\,, respectively.

Let us prove that if $\hat\lambda_0$ admits a density with respect
to Lebesgue measure, then $\hat\lambda_t$ does too. Let
$\mathcal{N}$ denote the set of Lebesgue measurable
 functions on $\R_+^\ell$
 that are not greater than one in absolute value and
are  equal to zero a.e. with respect to the Lebesgue measure.
Let $f\in\mathcal{N}$\,.
By the uniqueness  of solutions to \eqref{eq:44},
 $\psi_{s,t}^{-1}(u)$ is well defined, satisfies the
initial condition 
$\psi_{s,s}^{-1}(u)=u$ and the
version of
 \eqref{eq:44} in reverse time
 \begin{equation*}
       \partial_t\psi_{s,t}^{-1}(u)=-
F_t(\psi_{s,t}^{-1}(u))\,.
 \end{equation*}
Therefore, $\psi_{s,t}^{-1}(u)$ is of class $C^1$ in $u$\,.
 Since $\{u:\,U_{s,t}f(u)\not=0\}=\psi_{s,t}^{-1}\{u:\,f(u)\not=0\}$ and
 sets of  Lebesgue measure zero  are
preserved under $C^1$--maps, see, e.g., Lemma 1.1 on p.68 in 
Hirsch \cite{Hir94}, $U_{s,t}f=0$ a.e.
As $\hat\lambda_0$ admits a density, $\langle
U_{0,t}f,\hat\lambda_0\rangle=0$\,.
 By \eqref{eq:29}, for some $K>0$\,,
\begin{equation*}
\esssup_{f\in\mathcal{N}}\abs{    \langle f,\hat\lambda_t\rangle}\le
K\int_0^t
\esssup_{f\in\mathcal{N}}\abs{\langle U_{s,t}f,\hat
\lambda_s\rangle}\,ds
\le K\int_0^t
\esssup_{f\in\mathcal{N}}\abs{\langle f,\hat
\lambda_s\rangle}\,ds\,.
\end{equation*}By Gronwall's inequality,
 $\abs{\langle f,\hat
\lambda_s\rangle}=0$ when $f\in\mathcal{N}$\,, so, $\hat\lambda_s$
has a density which we denote by $\hat x_s(u)$\,.
Part 1 has been proved.

We prove part 2.
Since the $\beta^{k}(u)$ and $\delta^{k}(u)$ are bounded and
$\sup_{s\le t}\int_{\R_+^\ell}\abs{u}\hat\lambda_s(du)<\infty$\,, for all $t>0$\,,
by approximation, \eqref{eq:13} 
  holds for    $f(u)=\abs{u}$ so that
\begin{multline}
  \label{eq:56}
    \frac{d}{dt}\,\int_{\R_+^\ell}\abs{u}\hat\lambda_t(du)=
\int_{\R_+^\ell}\Bl(u\cdot\bl(\hat\beta(u)-\hat\delta(u)\br)
+
\int_{\R_+^\ell}
\abs{u'}\hat\varphi(u,du')\\
-\abs{u}\bl(\hat\phi(u)+\hat\epsilon(u)
\hat\lambda_t(\R_+^\ell)\br)\Br)\hat\lambda_t(du)\,.
\end{multline}
By \eqref{eq:91}, 
$ \int_{\R_+^\ell}\abs{u'}\hat\varphi(u,du')=\abs{u}\hat\phi(u)$\,.
Substitution in \eqref{eq:56} yields
\begin{equation*}
  \frac{d}{dt}\,\int_{\R_+^\ell} \abs{u}\, \hat\lambda_t(du) =
\int_{\R_+^\ell}
\, u\,\cdot(\hat\beta(u)-\hat\delta(u))
 \hat\lambda_t(du) 
-\hat\lambda_t(\R_+^\ell) \,\int_{\R_+^\ell}
\abs{u}\,\hat\epsilon(u)\hat \lambda_t(du)\,.
\end{equation*}
On the other hand,    \eqref{eq:12} can be written as
\begin{multline*}
  \int_{\R_+^\ell}\abs{u}\Lambda^{n,m}_t(du)=
\int_{\R_+^\ell}\abs{u}\Lambda^{n,m}_0(du)
+\int_0^t
\int_{\R_+^\ell}u\cdot\bl(\hat\beta^{n,m}(u)
-\hat\delta^{n,m}(u)\br)
\Lambda^{n,m}_s(du)\,ds
\\-\int_0^t\Lambda^{n,m}_s(\R_+^\ell)
\int_{\R_+^\ell}\abs{u}\hat\epsilon^{n,m}(u)
\Lambda^{n,m}_s(du)\,ds
+\frac{1}{mn}\, \overline{N}^{n,m}_t\,,
\end{multline*}
where
$\sup_{t\le L}
\abs{\overline N^{n,m}_t}/(mn)\to0$ in probability, as
$m,n\to\infty$\,, the latter convergence being proved in analogy with
\eqref{eq:62}. 
By
Lemma \ref{le:numberofgroups} and Lemma \ref{le:yule},
 in
probability, 
\begin{multline*}
  \int_{\R_+^\ell}\abs{u}\Lambda^{n,m}_t(du) 
\to \int_{\R_+^\ell}\abs{u}\hat \lambda_0(du)
+\int_0^t\int_{\R_+^\ell}u\cdot (\hat\beta(u)
-\hat\delta(u))
\hat \lambda_s(du) \,ds
\\-\int_0^t\hat \lambda_s(\R_+^\ell)
\int_{\R_+^\ell}\abs{u}\hat\epsilon(u)
\hat \lambda_s(du) \,ds\,.
\end{multline*}
Therefore, in probability in $\mathbb D(\R_+,\R)$\,,
\begin{equation*}
    \int_{\R_+^\ell}\abs{u}\Lambda^{n,m}_t(du) 
\to \int_{\R_+^\ell}\abs{u}\hat\lambda_t(du) \,.
\end{equation*}
By Lemma \ref{le:yule},  the  convergence holds locally
uniformly in $t$\,.
The other assertion of part 2 is proved similarly.
Part 2 has been proved.

We address now the regularity properties of $\hat x_s$\,, so, we
assume the hypotheses of part 3 of the theorem to hold.
Since $u\to \psi_{s,t}(u)$ is a diffeomorphism, given 
 $f\in \mathbb L^\infty(\R_+^\ell)$ and $z\in \mathbb
L^1(\R_+^\ell)$\,, 
by a change of variables, see, e.g., Theorem 2.6 on
p.505 in Lang \cite{Lan93}, 
$\langle U_{s,t}f,z\rangle_1=\langle f,z\!\circ\! \psi_{s,t}^{-1}\,
J(\psi_{s,t}^{-1})\rangle_1$\,, where
$\langle\cdot,\cdot\rangle_1$ represents the pairing between 
$\mathbb L^\infty$ and $\mathbb L^{1}$ and $J(\psi_{s,t}^{-1})$ denotes
the absolute value of the Jacobian determinant of $\psi_{s,t}^{-1}$\,.
Let us denote 
\begin{equation}
  \label{eq:31}
U^\ast_{s,t}z=z\!\circ\! \psi_{s,t}^{-1}\,
J(\psi_{s,t}^{-1})\,.
\end{equation}
It is a bounded operator on $\mathbb W^{1,\infty}(\R_+^\ell)$\,, with
norms being  bounded locally uniformly, cf., Proposition 9.6 on p.270
in Brezis \cite{Bre11}. (The $u$--derivatives of
$\psi^{-1}_{s,t}(u)$ are bounded locally uniformly in  $s$ and $t$\,.)
By \eqref{eq:34}, the  $B(\hat\lambda_s)$ are
 bounded operators on $\mathbb
W^{1,\infty}(\R_+^\ell)$ too, with locally uniformly bounded
norms.

By \eqref{eq:29},
for $f\in \mathbb L^\infty(\R_+^\ell)$\,,
\begin{equation*}
      \langle f,\hat x_t\rangle_1=\langle
  f,U_{0,t}^\ast\hat x_0\rangle_1
+\int_0^t
\langle f,U_{s,t}^\ast B(\hat\lambda_s)\hat x_s\rangle_1\,ds\,,
\end{equation*}
so, for almost all $u$\,,
\begin{equation}
  \label{eq:58}
    \hat x_t(u)=U_{0,t}^\ast\hat x_0(u)
+\int_0^tU_{s,t}^\ast B(\hat\lambda_s)\hat x_s(u)\,ds\,.
\end{equation}
We can therefore redefine $\hat x_t(u)$ as the latter righthand side,
which
makes it a continuous function of $t$ for all $u\in\R_+^\ell$\,.
Continuity of $\hat\lambda_t(\R_+)$ implies, as in the proof of 
 Scheffe's theorem,  that
 $\hat x_t $ is continuous in $t$ in $\mathbb
L^1(\R_+^\ell)$ for the strong topology. In particular, $\int_0^tU_{s,t}^\ast
B(\hat\lambda_s)\hat x_s\,ds$ is well defined as a Riemann integral in
$\mathbb L^1(\R^\ell)$ and
\begin{equation}
  \label{eq:85}
      \hat x_t=U_{0,t}^\ast\hat x_0
+\int_0^tU_{s,t}^\ast B(\hat\lambda_s)\hat x_s\,ds\,.
\end{equation}
 As in the proof of Lemma 4.5 on p.142 in Pazy \cite{Paz83}, see also
 Theorem 9.19 on p.488 in Engel and Nagel \cite{EngNag00}, there
exists  family  $V_{t}$ of bounded linear operators
on $\mathbb
W^{1,\infty}(\R_+^\ell)$  such that, for  $z\in\mathbb
W^{1,\infty}(\R_+^\ell)$\,,  
\begin{equation}
  \label{eq:39}
  V_{t}z=  U_{0,t}^\ast z+\int_0^t
 U_{s,t}^\ast B(\hat\lambda_s)
V_{ s}z\,d s\,.
\end{equation}
Specifically, one defines
\begin{align*}
V^{(0)}_{t}z=U_{0,t}^\ast z\,,\quad
    V^{(m)}_{t}z&=\int_0^t U^\ast_{s,t} B(\hat\lambda_s)
V_{s}^{(m-1)}z\,d s \intertext{and}
  V_{t}z&=\sum_{m=0}^\infty V^{(m)}_{t}z\,,
\end{align*}
the convergence holding in $\mathbb W^{1,\infty}(\R_+^\ell)$ because, as
induction shows, $\norm{V^{(m)}_t}\le K_1\\K_2^mt^m/m!$\,, where 
$K_1$ is an upper bound for $\norm{U_{0,t}^\ast}$ and $K_2$ is an
upper bound for
 $\norm{U_{s,t}^\ast}\norm{B(\hat\lambda_s)}$\,.
Let $\overline x_t=V_{t}\hat x_0$\,. By \eqref{eq:85},
\eqref{eq:39}, and the uniqueness of $\hat\lambda_t$\,,
$\overline x_t=\hat x_t$ as elements of $\mathbb L^1(\R_+^\ell)$\,, so $\hat x_t\in\mathbb
W^{1,\infty}(\R_+^\ell)$\,. By \eqref{eq:31}, $U^\ast_{s,t}z(u)$ is
Lipschitz--continuous   with respect to $t$\,, provided $z\in\mathbb
W^{1,\infty}(\R_+^\ell)$\,.
 By \eqref{eq:58}, $\hat x_t(u)$ is Lipschitz--continuous with respect
 to $t$\,.
Now \eqref{eq:28} is obtained from \eqref{eq:13} via integration by parts.
Part 3 has been proved.
\begin{remark}
  One can  see  that $U_{t,s}=U_{t,r}\circ U_{r,s}$ when
$s\le r\le t$\,, so, $U_{t,s}$ is an evolution system on
$\mathbb \R^{\R_+^\ell}$\,, see Engel and Nagel \cite{EngNag00} or 
Pazy \cite{Paz83} for the definitions.
\end{remark}
\section{Proof of Theorem \ref{the:lln2}}
\label{sec:proof-theor-refth}

The proof proceeds along similar lines to the one of Theorem
\ref{the:lln}. We give the main points.
Let, in analogy with \eqref{eq:19},
\begin{equation*}
  \check R^{n,m}_t=\frac{1}{mn^\ell}\,\sum_i X^{n,m}_t(i)=\int_{\R_+^\ell}\hat X^{n,m}_t(u)\,du\,.
\end{equation*}
The next three lemmas are proved similarly to 
Lemma \ref{le:numberofgroups},
 Lemma \ref{le:yule}, and Lemma \ref{le:dyra}, respectively.
\begin{lemma}
  \label{le:numberofgroups2}
 The sequence $(\check R^{n,m}_t\,,t\ge0)$ is $C$--tight
and, given $t>0$\,, there exists $\rho>0$ such that
$\mathbf P(\inf_{s\le t}\check R^{n,m}_s>\rho)\to1$\,, as $n,m\to\infty$\,.
\end{lemma}
\begin{lemma}
  \label{le:yule2}
The sequence  $\bl(\int_{\R_+^\ell}\abs{u}
\hat X^{n,m}_t(u)\,du\,, t\ge0\br)$ is $C$--tight.
\end{lemma}
\begin{lemma}
  \label{le:dyra1}
For all $L>0$\,,
\begin{equation*}\lim_{K\to\infty}
\limsup_{n,m\to\infty}
\mathbf P(\sup_{t\in[0,L]}\sum_i\abs{ X^{n,m}_t(i)}^2
>Km^2n^\ell)=0\,.
\end{equation*}
\end{lemma}
In analogy with \eqref{eq:9},
\begin{equation}
  \label{eq:93}
    \check Y_t^{n,m}=\check Y^{n,m}_0+\int_0^t \check Z^{n,m}_s\,ds
+\check N'^{n,m}_t\,,
\end{equation}
where
\begin{equation}
  \label{eq:94}
  \check  Y^{n,m}_t=
\int_{\R_+^\ell}
a^n(u)
\hat X^{n,m}_t(u)\,du\,,
\end{equation}

\begin{multline*}
  \check  Z^{n,m}_s=
\int_{\R_+^\ell}\Bl(
\sum_{k=1}^\ell\Bl(
n\bl(a^n\bl(u+\frac{e_k}{n}\br)-a^n(u)\br)
\,\hat X^{n,m}_s(u)\,
\frac{\lfloor nu_k\rfloor}{n}\,\hat\beta^{n,m,k}(u)\\+
n\bl(
a^n\bl(u-\frac{e_k}{n}\br)-
a^n(u)\br)\,\hat X^{n,m}_s(u)\,
\frac{\lfloor nu_k\rfloor}{n}\,(\hat\delta^{n,m,k}(u)+
\hat\mu^{n,m,k}(u))
\\+
n\bl(a^n\bl(u+\frac{e_k}{n}\br)-a^n(u)\br)
\hat X^{n,m}_s(u)
\displaystyle\int_{\R_+^\ell}
\frac{\hat X^{n,m}_s(u')}{ \check R^{n,m}_s}\,
\,\frac{\lfloor nu_k'\rfloor}{n}\,
\hat\mu^{n,m,k}(u')
\,du'\Br)
\\+\hat X^{n,m}_s(u)\,
\int_{\R_+^\ell}a^n(u')
\hat\phi^{n,m}(u)\check\eta^{n,m}(u,u')
\,du'
-a^n(u)\,\hat X^{n,m}_s(u)\,
\bl(\hat\phi^{n,m}(u)+
\check R^{n,m}_s\check\epsilon^{n,m}(u)\br)\Br)\,du
\\
+\int_0^t\int_{\R_+^\ell}\hat X^{n,m}_s(u)\,
\int_{B_{1/n}(\lfloor nu\rfloor/(2n))}a^n(u')
\hat\phi^{n,m}(u)\check\eta^{n,m}(u,u')
\,du'\,du\\
-\sum_i\sum_{k=1}^\ell
 \bl(a\bl(\frac{i+e_k}{n}\br)-a(\frac{i}{n})\br)
\frac{1}{ R^{n,m}_s}\,
i_k\mu^{n,m,k}(i)\frac{X^{n,m}_s(i)^2}{m^2n^\ell}\end{multline*}
and $(\check N'^{n,m}_t\,,t\ge0)$ is a locally square integrable
 martingale with the
predictable quadratic variation process, cf. \eqref{eq:61},
\begin{multline*}
      \langle\check N'^{n,m}\rangle_t=  \frac{1}{m}
\int_0^t\int_{\R_+^\ell}\Bl(\sum_{k=1}^\ell\Bl(
\frac{1}{n^{\ell-1}}\,\bl(a^n\bl(u+\frac{e_k}{n}\br)-
a^n(u)\br)^2
\hat X^{n,m}_s(u)
\frac{\lfloor nu_k\rfloor}{n}\,\hat\beta^{n,m,k}(u)\\
+\frac{1}{n^{\ell-1}}\,
\bl(a^n\bl(u-\frac{e_k}{n}\br)
-a^n(u)\br)^2
 \hat X^{n,m}_s(u)
\frac{\lfloor nu_k\rfloor}{n}\,\hat\delta^{n,m,k}(u)\\
+n\bl(a^n\bl(u+\frac{e_k}{n}\br)
-a^n(u)\br)^2
\hat X^{n,m}_s(u)
\int_{\R_+^\ell\setminus B_{1/n}(\lfloor nu\rfloor/n)}
\frac{\hat X^{n,m}_s(u')}{\check R^{n,m}_s}\,
\frac{\lfloor nu_k'\rfloor}{n}\,\hat\mu^{n,m,k}(u')
\,du'\\
+\frac{1}{n^{\ell-1}}\,\bl(a^n\bl(u-\frac{e_k}{n}\br)
-a^n(u)\br)^2
\hat X^{n,m}_s(u)
\frac{\lfloor nu_k\rfloor}{n}\,\hat\mu^{n,m,k}(u)\bl(1-
\frac{\hat X^{n,m}_s(u)}{n^\ell\check R^{n,m}_s}\br)\\
+n\bl(a^n\bl(u+\frac{e_k}{n}\br)
-a^n(u)\br)\frac{\hat X^{n,m}_s(u)}{n^\ell\check R^{n,m}_s}
\\
\int_{\R_+^\ell\setminus B_{1/n}(\lfloor nu\rfloor/n)}
\bl(a^n\bl(u'-\frac{e_k}{n}\br)
-a^n(u')\br)
\hat  X^{n,m}_s(u')\frac{\lfloor
  nu_k'\rfloor}{n}\,\hat\mu^{n,m,k}(u')
\,du'\Br)
\\+\frac{1}{n^\ell}\,
\hat\phi^{n,m}(u)\hat X^{n,m}_s(u)
\int_{\R_+^\ell}a^n(u')^2
\check\alpha^{n,m}(u,u')\,du'\\
+\frac{1}{n^\ell}\,a^n(u)^2\hat X^{n,m}_s(u)\bl(
\check R^{n,m}_s\check\epsilon^{n,m}(u)+\hat\phi^{n,m}(u)\br)\\
+\frac{1}{n^\ell}\,\hat X^{n,m}_s(u)\hat\phi^{n,m}(u)
\mathbf E\bl(\sum_{i'}a\bl(\frac{i'}{n}\br)
\theta^{n,m}_{\lfloor nu\rfloor}(i',1)\br)^2\\
-\frac{1}{n^\ell}\,a^n(u)
\hat X^{n,m}_s(u)
\hat\phi^{n,m}(u)\int_{\R_+^\ell}
a^n(u')\check\eta^{n,m}(u,u')\,du'
\Br)du\,ds
\,,
\end{multline*}
where $\check\alpha^{n,m}(u,u')=n^\ell\alpha^{n,m}_{\lfloor nu\rfloor}(\lfloor
nu'\rfloor)$\,. By \eqref{eq:104},
$  \check\alpha^{n,m}(u,u')
\le b\check\eta^{n,m}(u,u')\,.
$ Besides, $\mathbf E\bl(\sum_{i'}a(i'/n)
\theta^{n,m}_{\lfloor nu\rfloor}(i',1)\br)^2\le b\sup_{u}a(u)^2
\int_{\R_+^\ell}\check\eta^{n,m}(u,u')\,du'$\,.
By Lemma \ref{le:numberofgroups2}, Lemma
\ref{le:yule2} and \eqref{eq:90},
$\langle \check N'^{n,m}\rangle_t\to0$ in probability, as
$n,m\to\infty$\,.
Therefore,  for $\gamma>0$\,,
\begin{equation}
  \label{eq:92}
    \lim_{n,m\to\infty}\mathbf P(\sup_{t\le L}\abs{\check N'^{n,m}_t}>\gamma)=0\,.
\end{equation}

The following analogue of Lemma \ref{le:L^1} holds.
\begin{lemma}
  \label{le:L^{1,2}}
The sequence $\hat X^{n,m}$ is $C$--tight 
  in $\mathbb D(\R_+,\mathbb L^2(\R_+^\ell)\cap \mathbb L^1(\R_+^\ell))$\,.
\end{lemma}
\begin{proof}
 Taking as $F$ in Theorem 4.6 in Jakubowski
\cite{Jak86}  the set of functions 
$f\to \int_0^L g(u)f(u)\,du$\,, where
$L>0$\,, $f\in \mathbb L^2(\R_+^\ell)\cap \mathbb L^1(\R_+^\ell)$ and $g(u)$ is continuously differentiable of compact
support, 
it is sufficient to prove 
that, for all $L>0$\,,
\begin{equation}
  \label{eq:47}
    \lim_{K\to\infty}\limsup_{n,m\to\infty}
\mathbf P(\sup_{t\in[0,L]}\int_{\R_+^\ell} \abs{\hat X^{n,m}_t(u)}^2
\,du>K)=0\,,
\end{equation}
\begin{equation}
  \label{eq:41}
    \lim_{K\to\infty}\limsup_{n,m\to\infty}
\mathbf P(\sup_{t\in[0,L]}\int_{\R_+^\ell} \abs{\hat X^{n,m}_t(u)}
\,du>K)=0
\end{equation}
and, for all differentiable $g$ of compact support  and all $\gamma>0$\,,
\begin{equation}
  \label{eq:49}
   \lim_{\chi\to0} \limsup_{n,m\to\infty}
\mathbf
P(\sup_{\substack{s,t\in[0,L]:\\\abs{s-t}\le\chi}}\abs{\int\limits_{\R_+^\ell}g(u)(
\hat X^n_s(u)-\hat X^n_{t}(u))\,du}>\gamma)=0\,.
\end{equation}
The convergence in  \eqref{eq:47} is the statement of Lemma
\ref{le:dyra1}. 
By Lemma \ref{le:yule2},
 for 
 any
$\gamma>0$\,, 
\begin{equation*}
    \lim_{K\to\infty}
\limsup_{n,m\to\infty}\mathbf P(\sup_{t\le L}\int_{\R_+^\ell}
\mathbf 1_{[K,\infty)}(\abs{u})
\hat X_t^{n,m}(u)\,du>\gamma)=0\,,
\end{equation*}
so, \eqref{eq:41} holds.
Similarly to the proof of Lemma \ref{le:L^1},
 \eqref{eq:49} follows from  \eqref{eq:93}, \eqref{eq:94},
 \eqref{eq:5},
\eqref{eq:46}, and
\eqref{eq:92}. 
\end{proof}
Let $(\check x_s\,,s\ge0)$ represent a limit point of
$\hat X^{n,m}$\,.
We identify it in a similar fashion to that in the proof
of Theorem \ref{the:lln}.
  As $n,m\to\infty$\,, along a subnet,
 in distribution in $\mathbb
D(\R_+,\R)$\,,
\begin{equation*}
  (\check  Y^{n,m}_s,
\check  Z_s^{n,m})\to
(\check Y_s,
\check Z_s)\,,
\end{equation*}
where
$\check Y_s=\int_{\R_+^\ell}
a(u)
\check x_s(u)\,du$ and 
\begin{multline*}
   \check  Z_s=\sum_{k=1}^\ell\Bl(\int_{\R_+^\ell}
\partial_{u_k}a(u)
\,\check x_s(u)u_k(\hat\beta^{k}(u)-\hat\delta^{k}(u)-\hat\mu^k(u))\,du\\
+\frac{\displaystyle\int_{\R_+^\ell} \check x_s(u)
u_k\,\hat\mu^{k}(u)\,du}{
\displaystyle\int_{\R_+^\ell} 
\check x_s(u)\,du}
  \int_{\R_+^\ell} \partial_{u_k} a(u)\check x_s(u)
\,du\Br)
\\+\int_{\R_+^\ell}
a(u)
\int_{\R_+^\ell}
\check x_s(u')\check\varphi(u',u)\,du'\,du
-\int_{\R_+^\ell}
a(u)\,\check x_s(u)\hat\phi(u)\,du\\-\int_{\R_+^\ell}\check x_s(u)\,du\int_{\R_+^\ell}
a(u)\,\check x_s(u)\check\epsilon(u)\,du\,.
\end{multline*}
Since $\check Z_s$ has continuous trajectories,
by \eqref{eq:9},  \eqref{eq:81}, \eqref{eq:8}, and \eqref{eq:68},
\begin{equation*}
  \frac{d}{dt}\,
 \int_{\R_+^\ell}a(u)\check x_t(u)\,du= \check Z_t\,.
\end{equation*}
In analogy with the derivation of \eqref{eq:45}, one obtains
\begin{multline}
  \label{eq:74}
  \frac{d}{dt}\,
 \int_{\R_+^\ell}a_t(u)\check x_t(u)\,du=
\int_{\R_+^\ell}\partial_t a_t(u)\check x_t(u)\,du\\
+\sum_{k=1}^\ell\Bl(\int_{\R_+^\ell}
\partial_{u_k}a_t(u)
\,\check x_t(u)u_k(\hat\beta^{k}(u)-\hat\delta^{k}(u)-\hat\mu^k(u))\,du\\
+\frac{\displaystyle\int_{\R_+^\ell} \check x_t(u)
u_k\,\hat\mu^{k}(u)\,du}{
\displaystyle\int_{\R_+^\ell} 
\check x_t(u)\,du}
  \int_{\R_+^\ell} \partial_{u_k} a_t(u)\check x_t(u)
\,du\Br)
\\+\int_{\R_+^\ell}
a_t(u)
\int_{\R_+^\ell}
\check x_t(u')\check\varphi(u',u)\,du'\,du
-\int_{\R_+^\ell}
a_t(u)\,\check x_t(u)\hat\phi(u)\,du\\-\int_{\R_+^\ell}\check x_t(u)\,du\int_{\R_+^\ell}
a_t(u)\,\check x_t(u)\check\epsilon(u)\,du\,.
\end{multline}
 Since $(\check x_t,\,t\ge0)\in \mathbb C(\R_+,\mathbb
L^2(\R_+^\ell)\cap \mathbb L^1(\R_+^\ell))$ and
the functions 
$u_k\hat\beta^k(u)$\,, $u_k\hat\delta^k(u)$\,, and $u_k\hat\mu^k(u)$
are bounded, 
\eqref{eq:74} holds  for $(a_t\,, t\ge0)\in\mathbb C(\R_+,\mathbb 
W^{1,\infty}( \R_+^\ell)+ \mathbb
W^{1,2}(\R_+^\ell))$\,.
One next introduces $\check F_t$ and 
$\check U_{s,t}$ in analogy with \eqref{eq:40}, \eqref{eq:44} and
\eqref{eq:67}. 
It  is an operator on $\mathbb
 L^\infty(\R_+^\ell)$\,.
Since the functions 
$u_k\hat\beta^k(u)$\,, $u_k\hat\delta^k(u)$\,, and $u_k\hat\mu^k(u)$
are bounded and Lipschitz--continuous,
  by \eqref{eq:44} and 
Theorem 3.1.1 on p.76 in Hille \cite{Hil69}, the
functions $\psi_{s,t}(u)$ are Lipschitz--continuous with respect to
$u$\,. 
Similarly, $\psi_{s,t}^{-1}(u)$ is Lipschitz--continuous as well.
By the argument of the proof of Proposition 9.6 on p.270 in Brezis
\cite{Bre11}, see also Problem 7.5 on p.174 in 
Gilbarg and Trudinger \cite{GilTru01}, $\mathbb W^{1,\infty}(\R_+^\ell)$ is invariant under
$\check U_{t,s}$\,. Furthermore,  $\check U_{t,s}$ is 
a bounded  operator  on $\mathbb W^{1,\infty}(\R_+^\ell)$ and,
given $f\in\mathbb W^{1,\infty}(\R_+^\ell)$\,,
$\check U_{t,s}f(u)$ is an absolutely continuous
 function of $s$  and
$      \partial_s \check U_{s,t}f(u)=-A(\tilde\lambda_s)\check U_{s,t}f(u)$ a.e.
As in the proof of Theorem \ref{the:lln}, one lets $a_s=\check U_{s,t}f$ in
\eqref{eq:74}
and obtains \eqref{eq:29} with $\tilde\lambda_s(du)=\check
x_s(u)\,du$\,.
As before, that implies that $\tilde\lambda_s$ is
specified uniquely, so, $\check x_s(u)$ is specified uniquely for almost
all $u$ with respect to the Lebesgue measure.

Let $\hat x_0\in\mathbb W^{1,\infty}(\R_+^\ell)$\,, 
let $\check \varphi(u,u')$ 
have Sobolev derivative $D_u\check\varphi(u,u')$
with respect to $u$ for almost all $u'$ such that
$\esssup_{u\in\R_+^\ell}\int_{\R_+^\ell}( \check
  \varphi(u,u')+
  \abs{D_u\check\varphi(u,u')})\,du'<\infty$\,, 
and let $\hat\phi(u)$ and $\check\epsilon(u)$ be elements of $\mathbb
W^{1,\infty}(\R_+^\ell)$\,.
The $u_k$--derivatives of  $\psi_{t,s}^{-1}(u)$ being
bounded implies that $U^\ast_{s,t}$ is well defined and is a bounded operator
on $\mathbb W^{1,\infty}(\R_+^\ell)$\,.
In analogy with \eqref{eq:58}, one obtains that
\begin{equation*}
  \check x_t=U_{0,t}^\ast\hat x_0
+\int_0^tU_{s,t}^\ast B(\hat\lambda_s)\check x_s\,ds\,.
\end{equation*}
The rest of the proof mimics the proof of
 Theorem \ref{the:lln}.

\appendix
\section{Appendix}  \begin{lemma}\label{le:split}
    Let $N=(N_t\,,t\ge0)$ be a Poisson process of rate $\lambda$ adapted to
    filtration $\mathbf F=(\mathcal{F}_t\,,t\ge0)$\,.  Let process 
$X=(X_t\,,t\ge0)$ be adapted to
    $\mathbf F$ as well. Let  $\xi(k,x)$ be bounded random
    variables that are 
    independent of $\mathcal{F}_{\tau_k-}$\,, with $\tau_k$ denoting
    the $k$th jump of $N_t$\,. Let $Y_t=\int_0^t
    \xi(N_s,X_{s-})\,dN_s$
and let  $\mathbf{G}=(\mathcal{G}_t\,,t\ge0)$ be the smallest filtration that contains
 $\mathbf{F}$ such that $Y=(Y_t\,,t\ge0)$ is $\mathbf G$--adapted.
Then the process $Y$ has $\mathbf
G$--compensator $\tilde Y_t=\int_0^t\mathbf
E\xi(N_s+1,x)\large|_{x=X_s}\,\lambda\,ds$\,.
The locally square integrable martingale $(Y_t-\tilde Y_t\,,t\ge0)$ has the process
$\bl(\int_0^t\mathbf
E\xi(N_s+1,x)^2\large|_{x=X_s}\,\lambda\,ds\,,t\ge0\br)$ as the predictable
quadratic variation process.
  \end{lemma}
  \begin{proof}
We prove the first claim.
    It suffices to show that, for arbitrary $\mathbf G$--stopping
    time $\tau$\,, $
\mathbf E \int_0^\tau \xi(N_s,X_{s-})\,dN_s=
\mathbf E\int_0^\tau\mathbf
    E\xi(N_s+1,x)\large|_{x=X_s}\,\lambda\,ds$\,.
    Since $\{\tau_k\le\tau\}\in \mathcal{G}_{\tau_k-}$\,, the latter
    $\sigma$--algebra and the $\xi(k,x)$ are independent and
    $X_{\tau_k-}$ is $\mathcal{G}_{\tau_k-}$--measurable,
\begin{multline*}
   \mathbf E \int_0^\tau \xi(N_s,X_{s-})\,dN_s=
 \sum_{k} \mathbf E\ind{\tau_k\le \tau}\xi(k,X_{\tau_k-})=
\sum_{k} \mathbf E(\ind{\tau_k\le \tau}\mathbf
E(\xi(k,X_{\tau_k-})|\mathcal{G}_{\tau_k-}))\\=\sum_{k}
\mathbf 
E\bl(\ind{\tau_k\le \tau}\mathbf
E(\xi(k,x)|\mathcal{G}_{\tau_k-})\large|_{x=X_{\tau_k-}}
\br)=\sum_{k}\mathbf E\bl(\ind{\tau_k\le \tau}\mathbf
E(\xi(k,x)\large|_{x=X_{\tau_k-}})
\br)\\=\mathbf E\int_0^\tau\mathbf
E\xi(N_{s-}+1,x)\large|_{x=X_{s-}}
\,dN_s=\mathbf E\int_0^\tau\mathbf
E\xi(N_{s}+1,x)\large|_{x=X_{s}}
\,\lambda ds\,.
\end{multline*}
In order to prove the second claim, we write by the It\^o formula for the
locally square integrable martingale $M_t=Y_t-\tilde Y_t$\,,
\begin{equation*}
  M_t^2=2\int_0^t M_{s-}\,dM_s+\sum_{s\le t}(\Delta M_s)^2
=2\int_0^t M_{s-}\,dM_s+\int_0^t\xi(N_s,X_{s-})^2\,dN_s\,.
\end{equation*}
By the first part of the proof, the compensator of 
$\int_0^t\xi(N_s,X_{s-})^2\,dN_s$ is given by 
$\int_0^t\mathbf
E\xi(N_{s}+1,x)^2\large|_{x=X_{s}}
\,\lambda ds$\,.
  \end{proof}
\def\cprime{$'$} \def\cprime{$'$} \def\cprime{$'$} \def\cprime{$'$}
  \def\cprime{$'$} \def\polhk#1{\setbox0=\hbox{#1}{\ooalign{\hidewidth
  \lower1.5ex\hbox{`}\hidewidth\crcr\unhbox0}}} \def\cprime{$'$}
  \def\cprime{$'$} \def\cprime{$'$} \def\cprime{$'$} \def\cprime{$'$}
  \def\cprime{$'$}

\end{document}